\documentclass[12pt,a4paper]{article}
\usepackage{latexsym,amssymb,amsmath,a4wide,theorem}
\usepackage[a4paper, total={6in, 10in}]{geometry}
\usepackage[utf8]{inputenc}
\usepackage[T1]{fontenc}
\usepackage{times}
\usepackage{cite}
\usepackage{amsfonts}
\usepackage{graphicx}
\usepackage[all]{xy}
\usepackage{esint}
\usepackage[titletoc,title]{appendix}
\usepackage{indentfirst}
\usepackage{float}
\usepackage{cancel}
\usepackage{booktabs}
\usepackage{makecell}
\usepackage{color}
\usepackage{epstopdf}
\usepackage{hyperref}
\usepackage{mathtools} 
\usepackage{mathrsfs}
\usepackage{amsmath} 

\def\sqr#1#2{{\vbox{\hrule height.#2pt
     \hbox{\vrule width.#2pt height#1pt \kern#1pt
           \vrule width.#2pt}
     \hrule height.#2pt}}}
\def\square{\sqr74}
\def\qed{{\unskip\nobreak\hfil\penalty50\hskip1em
             \hbox{}\nobreak\hfil\square \parfillskip=0pt
             \finalhyphendemerits=0 \par\goodbreak \vskip8mm}}

\newenvironment{proof}{\medskip\par\noindent{\bf Proof\/}.\quad}{\qquad
\raisebox{-0.5mm}{\rule{1.5mm}{1mm}}\vspace{6pt}}

\usepackage{mathrsfs}

\newcommand{\D}{\mathrm{~d}}

\newtheorem{Thm}{Theorem}[section]

\newtheorem{Lem}{Lemma}[section]
\newtheorem{Prop}{Proposition}[section]

\newtheorem{Rem}{Remark}

\numberwithin{equation}{section}

\title{{ Multiple standing waves of Helmholtz equation with mixed dispersion concentrating in the high frequency limit}\thanks{The third author was supported by NSFC of China (Grant No. 12261107), and the fourth author was supported by CNPq/Brazil (Grant No. 304627/2023-2).
}}

\author{Shaoxiong Chen$^{1,2}$,~Fei Yuan$^1$,~Fukun Zhao$^{1,2}$\thanks{Corresponding authors.
		E-mail addresses: fukunzhao@163.com (F. Zhao).}~and~Jiazheng Zhou$^3$\\
	{\small 1. Department of Mathematics, Yunnan Normal University, Kunming 650500, PR China}\\
	{\small 2. Yunnan Research Center Of Modern Analysis And Partial Differential Equations,}\\ {\small Kunming 650500, PR China}\\
	{\small 3. Departamento de Matemática, Universidade de Brasília, Brasília, DF 70910-900, Brasil}}
\date{}

\begin{document}
\maketitle
\begin{abstract}
In this paper, we study the nonlinear Helmholtz equation with mixed dispersion
\begin{equation*}
	\Delta^2 u-\beta k^2\, \Delta u+\alpha k^4 u=W(x)\, |u|^{p-2}u~\text{in}~\mathbb{R}^N,
\end{equation*}
where the weight function $W(x)$ is continuous, nonnegative, and satisfies
\[
\limsup_{|x|\to\infty} W(x) \;<\; \sup_{x\in\mathbb{R}^N} W(x).
\]
Within each of the following parameter ranges,
\begin{center}
	(a) $\alpha<0$, $\beta\in\mathbb{R}$; \qquad (b) $\alpha>0$, $\beta<-2\sqrt{\alpha}$; \qquad  (c) $\alpha=0$, $\beta<0$,
\end{center}
After a suitable rescaling, we obtain the existence of dual ground state solutions, which concentrate along the global maximizers of $W$ as $k\to\infty$. In addition, we establish the existence of multiple solutions associated with the set of global maximum points of $W$, and we further characterize the precise concentration behavior of these solutions.

\end{abstract}

{\small \noindent{\bf  Keywords.}~Helmholtz equation; Mixed dispersion;  Concentration; Dual variational method; Dual ground state solution.}

\section{Introduction}
This paper is concerned with the existence, multiplicity and concentration behavior of solutions to the following nonlinear Helmholtz equation with mixed dispersion
\begin{equation}\label{1.1}
	\Delta^2 u - \beta k^2\, \Delta u + \alpha k^4 u
	= W(x)\, |u|^{p-2}u
	\quad \text{in } \mathbb{R}^N.
\end{equation}
where $\alpha,\beta \in \mathbb{R}$ and $p > 2$, the term containing $\Delta^2$ describes the higher order dispersion. 
When $k=1$ and $W(x)\equiv 1$, observe that taking the rescaling $w(x)=u(\beta^{\frac{1}{2}}x)$, \eqref{1.1} is equivalent to 
\begin{equation}\label{1.2}
	\gamma \Delta^2 w - \Delta w + \alpha  w =  |w|^{p-2} w
	\quad \text{in } \mathbb{R}^N, 
\end{equation}
where $\gamma=\frac{1}{\beta^2}$. In \cite{2000-shagalov&karpman-PD}, Karpman and Shagalov
 investigated the stabilizing role of higher-order dispersive effects using a comparatively simple model described by the equation
\begin{equation*}
  i\partial_{t}\Psi-\gamma \Delta^2\Psi +\Delta\Psi+|\Psi|^{2p}\Psi=0,~\Psi(0,x)=\Psi_0(x),~(t,x)\in\mathbb{R}\times\mathbb{R}^N.
\end{equation*}
The presence of the fourth order dispersion suggests stability when $(p-2)N<8$, whereas instability may occur when $(p-2)N\ge 8$. For analyses of well-posedness and scattering, see Ben-Artzi, Koch, and Saut~\cite{2000-artzi&koch&saut-PDE} as well as \cite{2007-Benoit-DPDE,2009-Pausader-JFA,2009-Pausader-DCDS}. Results on finite-time blow-up can be found in Boulenger and Lenzmann \cite{2017-Boulenger&Lenzmann-ASN} and the references therein. Moreover, \cite{2017-Boulenger&Lenzmann-ASN} also shows how the forth order Schr\"{o}dinger equation can be derived from the nonlinear Helmholtz equation via the so-called paraxial approximation. In addition, one dimensional standing waves for nonlinear Schr\"{o}dinger equations with mixed dispersion are of significant interest in water wave theory~\cite{1995-Buffoni-JDE,1996-Buffoni-NA}.
\par
We should point out that the parameter $\alpha$ appeared in \eqref{1.2} has important influence on the type of equation. When $\alpha=k^2\geq 0$ and $\gamma=0$, \eqref{1.2} is in fact a Schr\"{o}dinger equation (the case $\alpha=0$ corresponding to the zero mass case)
\begin{equation*}
  - \Delta w + k^2 w =  W(x) |w|^{p-2} w
	\quad \text{in } \mathbb{R}^N,
\end{equation*}
which is related to the standing wave solutions $v(t, x) = e^{ik^2 t} w(x)$ of the nonlinear time-dependent Schrodinger equation
\begin{equation*}
	i v_t + \Delta v +  W(x) |v|^{p-2} v = 0.
\end{equation*}
It is well known that the above stationary Schr\"{o}dinger equation has variational structure, i.e., its weak solutions corresponding to the critical points of the energy functional
\begin{equation*}
E(u)=\frac{1}{2}\int_{\mathbb{R}^N}(|\nabla u|^2+k^2u^2)dx-\frac{1}{p+1}\int_{\mathbb{R}^N}W(x)|u|^{p+1} \D x,~u\in H^1(\mathbb{R}^N).
\end{equation*}
\par
However, when $\gamma = 0$ and $\alpha = -k^2<0$, \eqref{1.2} reduces to the Helmholtz equation
\begin{equation}\label{1.3}
	-\Delta u - k^2 u =  W(x) |u|^{p-2} u~\text{in}~\mathbb{R}^N, 
\end{equation}
which can be viewed as derived from the standing-wave ansatz
\[
\psi(t, x) := e^{i\sqrt{k^2}t} u(x)
\]
for the nonlinear wave equation
\[
\partial_{tt}\psi - \Delta\psi =  W(x) |\psi|^{2} \psi,~(t,x)\in \mathbb{R}\times\in \mathbb{R}^N.
\]
Therefore, it is natural to attempt to find weak solutions of \eqref{1.3} by means of variational methods.
However, by the classical results of Rellich~\cite{1943-Rellich-JDMV} and Kato~\cite{1959-Kato-CPAM}, it is known that solutions of the Helmholtz equation \eqref{1.3} decay at most like
\[
u(x) = O\!\left(|x|^{-\frac{N-1}{2}}\right)
~\text{as}~|x| \to \infty.
\]
Consequently, for solutions of~\eqref{1.3} one can only expect
\[
u \in L^{p}(\mathbb{R}^{N})
~\text{and}~u \in W^{2,p}_{\mathrm{loc}}(\mathbb{R}^{N})
~\text{for}~p > \frac{2N}{N-1}.
\]
Therefore, one can not seek for solutions of \eqref{1.3} in $H^1(\mathbb{R}^N)$. Moreover, the spectrum $\sigma(-\Delta-k^2)=[-k^2,+\infty)$ and hence $0$ becomes an interior point of essential spectrum. These features bring challenges in constructing a variational framework.
\par
To overcome these difficulties, Ev\'equoz and Weth developed two kinds of useful methods. The one is reduce the equation \eqref{1.3} to a bounded domain $\Omega\subset\mathbb{R}^N$ by assuming $supp(W)\subset\Omega$, see \cite{2014-Evequoz&Weth-ARMA} for more general nonlinearity $f(x,u)$. The other one is the dual variational method proposed in \cite{2015-Evequoz&Weth-AIM}, which effectively circumvents the lack of a direct variational formulation.
More precisely, let $p^\prime = \frac{p}{p - 1}$ and denote by $\Phi_k$ the outgoing fundamental solution associated with the operator $-\Delta - k^2$.
Define the real part of the Helmholtz resolvent operator by
\begin{equation*}
	\mathbf{R}_k f := \operatorname{Re}\big(\Phi_k * f\big).
\end{equation*}
Then the problem can be equivalently rewritten as an integral equation
\begin{equation*}
	|v|^{p^\prime-2}v =  W^{1/p}\, \mathbf{R}_k\!\big( W^{1/p}v\big),
\end{equation*}
whose solutions correspond exactly to weak solutions of the original equation.
So they only need to seek for critical points of the dual energy functional $J_k : L^{p^\prime}(\mathbb{R}^N) \to \mathbb{R}$ defined by
\begin{equation*}
	J_k(v) = \frac{1}{p^\prime} \int_{\mathbb{R}^N} |v|^{p^\prime}\,\mathrm{d}x
	- \frac{1}{2} \int_{\mathbb{R}^N} \big( W^{1/p} v\big)\,
	\mathbf{R}_k\!\big( W^{1/p} v\big)\, \mathrm{d}x.
\end{equation*}
By means of the Mountain Pass Theorem, they proved that if
\begin{equation*}
	\frac{2(N+1)}{N-1} < p < \frac{2N}{N-2}, \quad (N \ge 3),
\end{equation*}
and $ W$ is positive, bounded, and $\mathbb{Z}^N$-periodic, then there exists a ground state solution
$v \in L^{p^\prime}(\mathbb{R}^N)$.
Consequently, the associated function
\begin{equation*}
	u := \mathbf{R}_k( W^{1/p}v)
\end{equation*}
belongs to $W^{2,q}(\mathbb{R}^N) \cap C^{1,\alpha}(\mathbb{R}^N)$ for any $q \in [p, \infty)$ and $\alpha \in (0, 1)$.
Moreover, under proper assumptions on $W$,
then, as $k \to \infty$, there exist solutions which, after suitable rescaling, tend to concentrate around the global maximizers of $ W$ (see~\cite{2017-Evequoz-AMPA}).
For further developments on Helmholtz-type equations, we refer the reader to~\cite{2021-Chen&Evequoz&Weth-SJAM,2021-Cossetti&Mandel-JFA,2015-Evequoz-ZAMP,
2014-Evequoz&Weth-ARMA,2020-Evequoz&Yecsil-PRSESA,
2022-Griesmaier&Knoller&Mandel-JMAA,2023-Guan&Murugan&Wei-CCM,
2024-liu&qiu&zhao-AML,2019-Mandel-ANS,2017-Mandel&Montefusco&Pellacci-ZAMP,
2021-Mandel&Scheider,2021-Mandel&Scheider&Yecsil-CVPDE}. 
The works \cite{2023-Ding&Wang-JDE,2025-Qiu&Yuan&Zhao-ZAMP} concerned with the Helmholtz systems.
\par
As far as we known, the equation \eqref{1.1} was only studied by Bonheure, Casteras, and Mandel ~\cite{2019-Denis-Jean-Rainer-JLMS}, for the case that $k=1$ and $W(x)\in L^\infty(\mathbb{R}^N)$ is $Z^N$-peridoic with $\inf\limits_{\mathbb{R}^N} W>0$.
To establish the existence of solutions, Bonheure, Casteras, and Mandel ~\cite{2019-Denis-Jean-Rainer-JLMS} introduced, in the sense of the limiting absorption principle, the resolvent operator
\begin{equation}\label{1.4}
	\mathscr{R} := (\Delta^2 - \beta\Delta + \alpha)^{-1},
\end{equation}
to overcome the difficulty that the essential spectrum of $L:=\Delta^2-\beta \Delta+\alpha$ contains zero. Their analysis is based on a modification of the dual variational method developed by Evéquoz and Weth \cite{2015-Evequoz&Weth-AIM}. The central idea is to construct a resolvent operator whose mapping properties are comparable in several aspects even better than those available in the second order case.
In particular, since the term $\Delta^2 u$ accounts for higher order dispersion, the associated linear operator displays a significantly more intricate structure compared to the classical second order Helmholtz operator. By deriving precise mapping properties of \eqref{1.4} and establishing refined oscillatory decay estimates for its kernel, they successfully reconstructed a dual variational framework exhibiting the required Mountain Pass geometry. This in turn allows the Mountain Pass Theorem (\cite{1973-Ambrosetti-Rabinowitz-JFA}) to be applied within the modified dual setting, leading to the existence of nontrivial solutions to the fourth-order nonlinear Helmholtz equation. 
\par
However, when $W(x)$ is not $\mathbb{Z}^N$-periodic, the questions of existence, multiplicity and concentration behavior of solutions of \eqref{1.1} as $k\to\infty$ remain open, and are precisely the focus of the present work.
To state our main results, we make the following assumptions.
\begin{itemize}
  \item[$(A_1)$]
  \begin{itemize}
    \item[(1)] $N \ge 2$, $\alpha < 0$ and $\beta\in\mathbb{R}$, or $\alpha > 0$ and $\beta < -2\sqrt{\alpha}$.
    \item[(2)] $p\in (\frac{2(N+1)}{N-1},2^{**})$, where
\begin{equation*}
	2^{**} := \frac{2N}{(N-4)_+} :=
	\begin{cases}
		+\infty, & N \le 4,\\[2pt]
		\dfrac{2N}{N-4}, & N \ge 5.
	\end{cases}
\end{equation*}
  \end{itemize}   
\end{itemize}

\begin{itemize}
  \item[$(A_2)$]
  \begin{itemize}
    \item[(1)] $N \ge 3$ and $\alpha = 0,\ \beta < 0$.
    \item[(2)] $p\in (\frac{2N}{N-2},2^{**})$.
  \end{itemize}
\end{itemize}
  \begin{itemize}
	\item[($W_1$)] $W$ is continuous, bounded, and nonnegative in $\mathbb{R}^N$;
	\item[($W_2$)] $W$ satisfies
	\begin{equation}\label{1.5}
		W_\infty:=\limsup_{|x|\to\infty}W(x)\;<\;W_0:=\sup_{x\in\mathbb{R}^N}W(x).
	\end{equation}
\end{itemize}
\par
The first result concerning with the existence and concentration behavior of the dual ground state solutions.
\begin{Thm}\label{thm1.1}
	Assume $(A_1)$ or $(A_2)$, and $W(x)$ satisfies $(W_1)$ and $(W_2)$. Then
	\begin{itemize}
		\item[(i)] there exists $k_0>0$ such that, for all $k>k_0$, \eqref{1.1} admits a dual ground state solution;
		\item[(ii)] let $k_n\to \infty$ with $k_n>k_0$ for all $n$, and for each $n$ let $u_n$ be a dual ground state solution of
		\begin{equation*}
			\Delta^2 u_n - \beta k_n^2\, \Delta u_n + \alpha k_n^4\, u_n
			= W(x)\, |u_n|^{p-2}u_n
			~\text{in}~\mathbb{R}^N.
		\end{equation*}
		Then there exists a point $x_0\in\mathbb{R}^N$ such that $W(x_0)=W_0$, a ground state solution $u_0$ of the limit problem
		\begin{equation*}
			\Delta^2 u - \beta\, \Delta u + \alpha\, u
			= W_0\, |u|^{p-2}u
			\quad \text{in } \mathbb{R}^N,
		\end{equation*}
		and a sequence $(x_n)$ with $x_n\to x_0$ such that, after translation and rescaling,
		\begin{equation*}
			u_n\!\left(\tfrac{\cdot}{k_n}+x_n\right) \to u_0
			\quad \text{in } L^p(\mathbb{R}^N)\ \text{as } n\to\infty.
		\end{equation*}
	\end{itemize}
\end{Thm}
\par
The strategy of the proof comes from \cite{2019-Denis-Jean-Rainer-JLMS} and  \cite{2015-Evequoz&Weth-AIM}, a key structural observation in the proof of Theorem \ref{thm1.1} is that the operator $L$ can be factorized into the product of two second-order operators which was found in \cite{2019-Denis-Jean-Rainer-JLMS}
\begin{equation}\label{1.6}
	L = \Delta^2 - \beta\Delta + \alpha = (-\Delta - a_1)(-\Delta - a_2),
\end{equation}
where
\begin{equation*}
	a_1 := \frac{-\beta + \sqrt{\beta^2 - 4\alpha}}{2}~\text{and}~
	a_2 := \frac{-\beta - \sqrt{\beta^2 - 4\alpha}}{2}.
\end{equation*}
Moreover, it is readily seen that, if $\alpha<0$, then $a_{1}>0>a_{2}$; while if $\alpha=0$ and $\beta<0$, one has $a_{1}>0=a_{2}$, and in these cases $L$ can be regarded as the composition of a Schrödinger operator with a Helmholtz (or Laplace) operator. In particular, if $\alpha>0$ and $\beta<-2\sqrt{\alpha}$, then $a_{1}>a_{2}>0$, so that $L$ factorizes into the product of two Helmholtz operators, whereas in the borderline case $\beta=-2\sqrt{\alpha}$ a double root occurs and we have $a_{1}=a_{2}=\sqrt{\alpha}$.
Thus, the above factorization shows that at least one factor is always of Helmholtz type, whereas the other factor switches between Schrödinger type and Helmholtz (or Laplace) type, depending on the parameter regime. Despite these structural differences of $L$ in different parameter regions, the construction of its fundamental solution can still be carried out within a unified analytical framework. More precisely, as will be shown in Section~2, one can derive a representation formula for the fundamental solution that is valid for all admissible parameter regimes.
\par
Under condition \eqref{1.5}, the dual energy functional depending on $k$ satisfies the Palais--Smale condition, and the energy level ``at infinity'' is strictly higher than the lowest energy level in a bounded region. Unlike in related problems (see, for instance, \cite{2003-Alves&Soares&Yang-ANS}), here we have no sign information on the nonlocal term $\int_{\mathbb{R}^N} u\mathscr{R}v dx$, since the real part of the Helmholtz resolvent is not positive definite. This difficulty is overcome by a new energy estimate (Lemma~\ref{lem:offdiag}) for the nonlocal interaction between functions with disjoint supports. The $L^p$-concentration in conclusion (ii) follows from energy comparison and a bubble decomposition for Palais--Smale sequences (Lemma~\ref{lem:profile}), inspired by the work of Benci and Cerami~\cite{1987-Benci&Cerami-ARMA}.
\medskip
\par
Next, we will focus on the multiplicity of solutions. Let
\begin{equation*}
	M:=\{x\in\mathbb{R}^N:\ W(x)=W_0\}
\end{equation*}
denote the set of global maxima of $W$. By $(W_1)$ and $(W_2)$, $M\neq \emptyset$. For $\delta>0$, set
\begin{equation*}
	M_\delta:=\{x\in\mathbb{R}^N:\ \operatorname{dist}(x,M)\le\delta\}.
\end{equation*}
If $Y$ is a closed subset of a metric space $X$, we denote by $\operatorname{cat}_X(Y)$ the Lusternik--Schnirelmann category of $Y$ relative to $X$ (see \cite{1988-Szulkin-Ann}).
\par
The second result implies the multiplicity of nontrivial solutions related to the topology of $M_\delta$.
\begin{Thm}\label{thm1.2}
Under the assumptions of Theorem \ref{thm1.1}, for every $\delta>0$,  there exists $k(\delta)>0$ such that, for all $k>k(\delta)$, problem~\eqref{1.1} has at least $m := \operatorname{cat}_{M_\delta}(M)$
pairwise distinct nontrivial solutions
\[
u_k^{(1)},\dots,u_k^{(m)}.
\]
Moreover, these solutions satisfy the following concentration property: for each fixed $j\in\{1,\dots,m\}$, for any sequence $(k_n)_n$ with $k_n>k(\delta)$ and $k_n\to\infty$, there exist $x_0^{(j)}\in M$, a ground state solution $u_0^{(j)}$ of the limit equation, and a sequence of points $(x_n^{(j)})_n\subset\mathbb{R}^N$ such that, up to a subsequence,
\[
x_n^{(j)}\to x_0^{(j)},\qquad
k_n^{-\frac{4}{p-2}}\,
u_{k_n}^{(j)}\!\Big(\tfrac{\cdot}{k_n}+x_n^{(j)}\Big)
\ \to\ u_0^{(j)}
\quad \text{in }L^p(\mathbb{R}^N)\ \text{as }n\to\infty.
\]
\end{Thm}

The strategy of the proof of Theorem \ref{thm1.2} is topological in nature and follows the scheme of Cingolani and Lazzo~\cite{1997-Cingolani&Lazzo-TMNA} (see also \cite{2000-Cingolani&Lazzo-JDE}), which in turn builds on ideas of Benci, Cerami, and Passaseo~\cite{1991-Benci&Cerami&Passaseo-NA,1991-Benci&Cerami-ARMA} in bounded domains. The crucial point is to construct two maps whose composition is homotopic to the inclusion $M\hookrightarrow M_\delta$. For related multiplicity results for Schrödinger equations with $V>0$ and small $\varepsilon>0$, we refer to \cite{2015-Cingolani&Jeanjean&Tanaka-CVPDE} and the references therein.

Once the multiplicity has been established, we then turn to the concentration behavior of the corresponding family of solutions. For this purpose, we adapt the method developed by Alves and Figueiredo for quasilinear problems in $\mathbb{R}^N$ with critical exponential growth~\cite{AlvesFigueiredo-crit-exp-RN}.
\par
This paper is organized as follows. In Section~2, for fixed $k$, we set up the dual variational framework, introduce the Nehari manifold, and derive a bubble decomposition for Palais--Smale sequences when $W$ is constant. Based on these tools, we show that the dual functional satisfies the Palais--Smale condition on the Nehari manifold below a suitable energy threshold. A key ingredient is an ``off-diagonal'' decay estimate (Lemma~\ref{lem:offdiag}) for the nonlocal interaction induced by the fourth-order Helmholtz resolvent.
In Section~3, in the limit $\varepsilon=k^{-1}\to 0$, we show that the least energy among critical points of the dual functional is achieved (Proposition~\ref{prop:33}), which yields Theorem~\ref{thm1.1}(i); we then establish the concentration of dual ground states (Proposition~\ref{prop:34}), thus proving Theorem~\ref{thm1.1}(ii).
The proof of Theorem~\ref{thm1.2} is given in Section~4.

\medskip

\textbf{Notation.}
\begin{itemize}
	\item Assume that $p>1$, and let $p'=\dfrac{p}{p-1}$ be the H\"{o}lder conjugate exponent of $p$. In what follows, unless the range of $p$ is specified otherwise, we always assume that $p$ satisfies either $(A_1)$ or $(A_2)$.
	
	\item We denote by $\|\,\cdot\,\|_q$ the norm in $L^q(\mathbb{R}^N)$ ($1\le q\le\infty$).
	\item $\mathcal{S}(\mathbb{R}^N)$ denotes the Schwartz space of rapidly decreasing functions on $\mathbb{R}^N$.
	\item For $r>0$ and $x\in\mathbb{R}^N$, $B_r(x)$ is the open ball centered at $x$ with radius $r$; we write $B_r:=B_r(0)$. $\mathbf{1}_B$ denotes the characteristic function of a set $B$.
	\item  According to \cite{1964bookGelfand}, for $\nu\in\mathbb{R}$ the Hankel function of the first kind is defined by
	\[
	H_\nu^{(1)}(z)=J_\nu(z)+ i\,Y_\nu(z),
	\]
	where $J_\nu$ is the Bessel function of the first kind,
	\[
	J_\nu(z)=\sum_{m=0}^{\infty}\frac{(-1)^m}{m!\,\Gamma(m+\nu+1)}
	\left(\frac{z}{2}\right)^{2m+\nu},
	\]
	and $\Gamma$ is the Gamma function
	\[
	\Gamma(s)=\int_{0}^{\infty} t^{s-1}e^{-t}\,dt,\qquad s>0.
	\]
	Moreover, $Y_\nu$ denotes the Bessel function of the second kind. For $\nu\notin\mathbb{Z}$,
	\[
	Y_\nu(z)=\frac{J_\nu(z)\cos(\pi\nu)-J_{-\nu}(z)}{\sin(\pi\nu)},
	\]
	and for $n\in\mathbb{Z}$ it is defined by continuity as
	\[
	Y_n(z):=\lim_{\nu\to n}Y_\nu(z).
	\]
Meanwhile, one also has
	\[
	\frac{1}{x+i0}:=\lim_{\varepsilon\to0^+}\frac{1}{x+i\varepsilon}
	=\operatorname{p.v.}\frac{1}{x}-i\pi\,\delta(x),
	\]
	where $\delta$ is the Dirac delta distribution and $\operatorname{p.v.}\frac{1}{x}$ denotes the Cauchy principal value.

\end{itemize}

\section{Variational framework}

\subsection{Variational framework}

Let $\varepsilon := k^{-1}$ and define
\[
u_\varepsilon(x) := \varepsilon^{\frac{4}{p-2}}\,u(\varepsilon x),~
W_\varepsilon(x) := W(\varepsilon x),
\]
then \eqref{1.1} can be rewritten as
\begin{equation}\label{2.1}
	\Delta^2 u_\varepsilon - \beta\, \Delta u_\varepsilon + \alpha\, u_\varepsilon
	= W_\varepsilon(x)\, |u_\varepsilon|^{p-2}u_\varepsilon
	\quad \text{in } \mathbb{R}^N.
\end{equation}
When $\alpha$ and $\beta$ satisfy
\begin{center}
	(a) $\alpha<0$, $\beta\in\mathbb{R}$; \qquad (b) $\alpha>0$, $\beta<-2\sqrt{\alpha}$; \qquad  (c) $\alpha=0$, $\beta<0$,
\end{center}
the essential spectrum of $L:=\Delta^2-\beta\Delta+\alpha$ includes $0$. This creates additional difficulties. To overcome this difficulty, Bonheure, Casteras, and Mandel adapted the dual variational method of Ev\'equoz and Weth \cite{2015-Evequoz&Weth-AIM} to \eqref{2.1}. As in the second-order case, they constructed a resolvent-type operator associated with $L$.

By the factorization (see (\ref{1.6})),
\[
L=\Delta^2-\beta\Delta+\alpha = (-\Delta-a_1)(-\Delta-a_2),
\]
so the analysis reduces to the Helmholtz operators of the form $-\Delta-a$ with $a>0$.
The outgoing fundamental solution $\Phi_a$ of $(-\Delta-a)u=\delta$ is given by
\begin{equation*}
	\Phi_a(x)
	= \frac{i}{4}\left(\frac{\sqrt{a}}{2\pi |x|}\right)^{\frac{N-2}{2}}
	H^{(1)}_{\frac{N-2}{2}}\!\left(\sqrt{a}\,|x|\right).
\end{equation*}
The associated resolvent  is
\begin{equation}\label{2.4}
	(\mathbf{R}_a f)(x)
	:= \lim_{\varepsilon\to0^+} (\mathbf{R}_{a+i\varepsilon} f)(x)
	= \frac{1}{(2\pi)^{N/2}}
	\int_{\mathbb{R}^N}\frac{\widehat{f}(\xi)}{|\xi|^2-(a+i0)}\,e^{ix\cdot\xi}\,\D \xi.
\end{equation}
In fact, \cite[Theorem~6]{2004-Gutierrez-MA} shows that
\[
\|\mathbf{R}_af\|_{L^{p}(\mathbb{R}^N)} \le C \|f\|_{L^{p^\prime}(\mathbb{R}^N)},  \quad f \in \mathcal{S}(\mathbb{R}^N),
\]
and therefore $\mathbf{R}_a$ extends to a bounded linear operator from $L^{p^\prime}(\mathbb{R}^N)$ to $L^{p}(\mathbb{R}^N)$.
Moreover,
\begin{equation}\label{2.5}
	(\mathbf{R}_a f)(x)=(\Phi_a*f)(x)
	=\frac{1}{a}\,\mathbf{R}_1\!\Big(f\!\left(\tfrac{\cdot}{\sqrt{a}}\right)\Big)\!(\sqrt{a}\,x),
	\qquad a>0.
\end{equation}

In accordance with the above factorization, for the fourth-order operator $L$ we introduce
\begin{equation}\label{2.6}
	G := \frac{1}{\sqrt{\beta^2-4\alpha}}\,(\Phi_{a_1}-\Phi_{a_2}),
	\qquad
	\widehat{G}(\xi)
	= \frac{1}{\sqrt{\beta^2-4\alpha}}
	\left(\frac{1}{|\xi|^2-a_1}-\frac{1}{|\xi|^2-a_2}\right).
\end{equation}
It is then straightforward to check that
\[
\Delta^2 G - \beta\,\Delta G + \alpha\, G = \delta,
\]
so that $G$ is a fundamental solution of $L$. Formally, the expression in \eqref{2.6} also makes sense in the regime $a_1,a_2<0$ (see \cite[Proposition~3.13]{2018-Bonheure-SJMA}); however, in the present paper we restrict our attention to the cases $a_1>0>a_2$, $a_1>a_2>0$, and $a_1>0=a_2$. By invoking standard properties of Hankel functions (see~\cite{2019-Denis-Jean-Rainer-JLMS}), we deduce the pointwise bounds
\begin{align}\label{guji}
	\begin{aligned}
		|G(x)| \le
		\begin{cases}
			C\big(|x|^{4-N} + |\log(|x|)|\big), & N \ge 4,\ 0<|x|\le 1, \\[4pt]
			C, & N \in \{2,3\},\ 0<|x|\le 1,\\
			C\,|x|^{\frac{1-N}{2}}, & N \ge 2,\ |x|\ge 1.
		\end{cases}
	\end{aligned}
\end{align}

Moreover, $G$ satisfies
\[
\nabla G(x) - i\sqrt{a_1}\, G(x)\, \frac{x}{|x|}
= o\!\left(|x|^{\frac{1-N}{2}}\right),
\qquad |x|\to\infty,\ a_1 > 0 > a_2,
\]
if $a_1>a_2>0$, as  $|x| \rightarrow \infty$, we have
\begin{equation*}
	\left|\nabla G(x)-\frac{i}{a_1-a_2}\left(\sqrt{a_1} \Phi_{a_1}(x)-\sqrt{a_2} \Phi_{a_2}(x)\right) \frac{x}{|x|}\right|
	= o\left(|x|^{\frac{1-N}{2}}\right) 
\end{equation*}
Note that in the case $a_1>0>a_2$, the functions $\Phi_{a_2}, \Phi_{a_2}^{\prime}$ decrease exponentially and hence much faster than $\Phi_{a_1}, \Phi_{a_1}^{\prime}$ at infinity, so that \eqref{2.6} allows one to deduce the Sommerfeld condition from \eqref{2.4}. Here, $\Phi_{a_2}$ denotes the Green's function of the Schrödinger operator $-\Delta-a_2$. As to the case $a_1>a_2>0$, let us remark that $\sqrt{\beta^2-4\alpha}=a_1-a_2$. Motivated by \eqref{2.4}--\eqref{2.6}, we may now define
\begin{equation}\label{2.7}
	\mathscr{R}f
	:= \lim_{\varepsilon\to0^+}\frac{1}{\sqrt{\beta^2-4\alpha}}
	\big(\mathbf{R}_{a_1+i\varepsilon}f - \mathbf{R}_{a_2+i\varepsilon}f\big).
\end{equation}
According to \cite[Theorems~3.3 and~3.4]{2019-Denis-Jean-Rainer-JLMS}, the operator $\mathscr{R}$ introduced in~\eqref{2.7} admits a bounded extension from $L^{p}(\mathbb{R}^{N})$ to $L^{q}(\mathbb{R}^{N})$ under the following conditions on $p$ and~$q$.

First, assume that $(A_1)$ holds. Then \(\mathscr{R}\) extends to a bounded linear operator
\[
\mathscr{R}:L^p(\mathbb{R}^N)\longrightarrow L^q(\mathbb{R}^N),
\]
that is, there exists a constant \(C>0\) such that
\begin{equation}\label{3.1}
	\|\mathscr{R}f\|_{L^q(\mathbb{R}^N)}
	\le C\,\|f\|_{L^p(\mathbb{R}^N)}
\end{equation}
for all \(f\in L^p(\mathbb{R}^N)\), provided that the exponents \(p,q\in[1,\infty]\) satisfy
\begin{equation}\label{3.2}
	\frac{2}{N+1}\ \le\ \frac1p-\frac1q\ \le\
	\begin{cases}
		1, & N\in\{2,3\},\\[2pt]
		<\,1, & N=4,\\[2pt]
		\dfrac{4}{N}, & N\ge5,
	\end{cases}
	\qquad
	\frac1p>\frac{N+1}{2N},\quad \frac1q<\frac{N-1}{2N}.
\end{equation}
In particular, if \(q=p'\) is the H\"older conjugate of \(p\), then~\eqref{3.1} holds whenever
\[
\begin{cases}
	\dfrac{2(N+1)}{N-1}\ \le\ q\ \le\ \infty,& N\in\{2,3\},\\[6pt]
	\dfrac{10}{3}\ \le\ q\ <\ \infty,& N=4,\\[8pt]
	\dfrac{2(N+1)}{N-1}\ \le\ q\ \le\ \dfrac{2N}{N-4},& N\ge5.
\end{cases}
\]

Next, assume that $(A_2)$ holds. Then \(\mathscr{R}\) admits a bounded extension
\[
\mathscr{R}:L^{p}(\mathbb{R}^{N})\longrightarrow L^{q}(\mathbb{R}^{N}),
\]
for every \(p,q\in[1,\infty]\) such that
\[
\frac{2}{N}\ \le\ \frac{1}{p}\ <\ 1-\frac{1}{q},
\qquad
\frac{1}{p}\ >\ \frac{N+1}{2N},\quad
\frac{1}{q}\ <\ \frac{N-1}{2N},
\]
and
\begin{equation}\label{3.20}
	1-\frac{1}{q}\ \le\
	\begin{cases}
		1, & N=3,\\[2pt]
		1, & N=4,\\[2pt]
		\dfrac{4}{N}, & N>5.
	\end{cases}
\end{equation}
In particular, for \(q=p'\) one obtains the following admissible ranges:
\[
\begin{cases}
	6\ \le\ q\ \le\ \infty, & N=3,\\[4pt]
	4\ \le\ q\ \le\ \infty, & N=4,\\[4pt]
	\dfrac{2N}{N-2}\ \le\ q\ \le\ \dfrac{2N}{N-4}, & N\ge5.
\end{cases}
\]

To establish the subsequent results, we adopt a dual variational framework.
For the scaled problem~\eqref{2.1}, set
\[
v := W_{\varepsilon}^{\,1/p'}\,|u_\varepsilon|^{p-2}u_\varepsilon
\in L^{p'}(\mathbb{R}^{N}),
\]
so that \(v\) satisfies the integral equation
\begin{equation}\label{2.9}
	W_{\varepsilon}^{\,1/p}\,\mathbf{R}\!\bigl(W_{\varepsilon}^{\,1/p} v\bigr)
	= |v|^{p'-2}v \quad \text{in }\mathbb{R}^{N}.
\end{equation}
Here \(\mathbf{R} := \operatorname{Re}(\mathscr{R})\) denotes the real part of the resolvent-type operator \(\mathscr{R}\) associated with the fourth-order operator \(L\) (see~\eqref{2.7}).

Arguing as in \cite[Proposition~4.1]{2019-Denis-Jean-Rainer-JLMS}, one obtains the symmetry identity
\begin{equation}\label{2.8}
	\int_{\mathbb{R}^N} (\mathbf{R} f)\, g\,\D x
	= \int_{\mathbb{R}^N} f\, (\mathbf{R} g)\,\D x,
	\qquad \forall\, f,g\in\mathcal{S}(\mathbb{R}^N).
\end{equation}
Let \(W\in L^{\infty}(\mathbb{R}^N)\) be nonnegative with \(W\not\equiv 0\).
We then introduce the Birman--Schwinger type operator
\[
K_\varepsilon:L^{p'}(\mathbb{R}^N)\to L^{p}(\mathbb{R}^N),\qquad
K_\varepsilon(v):=W_\varepsilon^{\,1/p}\,\mathbf{R}\!\bigl(W_\varepsilon^{\,1/p}v\bigr).
\]
By the symmetry property~\eqref{2.8}, the operator \(K_\varepsilon\) is symmetric with respect to the \(L^{p'}\)–\(L^{p}\) duality, in the sense that for all \(v,w\in L^{p'}(\mathbb{R}^N)\),
\[
\int_{\mathbb{R}^N} v\, K_\varepsilon(w)\,\mathrm{d}x
\;=\; \int_{\mathbb{R}^N} w\, K_\varepsilon(v)\,\mathrm{d}x .
\]
Moreover, \(K_\varepsilon\) is locally compact: for every bounded measurable set \(B\subset\mathbb{R}^N\), the truncated operator
\[
\mathbf{1}_B K_\varepsilon:\ L^{p'}(\mathbb{R}^N)\longrightarrow L^{p}(\mathbb{R}^N)
\]
is compact (see \cite[Proposition~4.1]{2019-Denis-Jean-Rainer-JLMS}).

Equation~\eqref{2.9} is precisely the Euler--Lagrange equation associated with the functional
\(
J_{\varepsilon}\in C^{1}\!\bigl(L^{p'}(\mathbb{R}^{N}), \mathbb{R}\bigr)
\)
defined by
\begin{equation}\label{2.10}
	J_{\varepsilon}(v)
	:=\frac{1}{p'}\int_{\mathbb{R}^{N}}|v|^{p'}\,\mathrm{d}x
	\;-\;\frac{1}{2}\int_{\mathbb{R}^{N}}v\,K_\varepsilon v\,\mathrm{d}x .
\end{equation}
Every critical point of \(J_\varepsilon\) gives rise to a solution of~\eqref{2.1}.
More precisely, for \(v \in L^{p'}(\mathbb{R}^N)\) one has
\begin{equation}\label{eq:EL}
	J_\varepsilon'(v)=0 \quad \Longleftrightarrow \quad
	|v|^{p'-2}v \;=\; W_\varepsilon^{\,1/p}\,\mathbf{R}\!\bigl(W_\varepsilon^{\,1/p}v\bigr).
\end{equation}
If we write \(v=W_\varepsilon^{\,1/p'}|u_\varepsilon|^{p-2}u_\varepsilon\), then~\eqref{eq:EL} is equivalent to the integral formulation
\begin{equation}\label{2.11}
	u_\varepsilon \;=\; \mathbf{R}\!\bigl(W_\varepsilon\,|u_\varepsilon|^{p-2}u_\varepsilon\bigr).
\end{equation}

Under $(A_1)$ or $(A_2)$, $J_\varepsilon$ has the mountain-pass geometry:
\begin{itemize}
	\item[(i)] There exist $\delta>0$ and $\rho\in(0,1)$ such that $J_\varepsilon(v)\ge \delta$ whenever $\|v\|_{L^{p^\prime}(\mathbb{R}^{N})}=\rho$.
	\item[(ii)] There exists $v_{0}\in L^{p^\prime}(\mathbb{R}^{N})$ with $\|v_{0}\|_{L^{p^\prime}(\mathbb{R}^{N})}>1$ and $J_\varepsilon(v_{0})<0$.
\end{itemize}
Accordingly, the associated Nehari set is
\[
\mathcal{N}_\varepsilon := \bigl\{\,v \in L^{p^\prime}(\mathbb{R}^N)\setminus\{0\}: \ J_\varepsilon'(v)\,v=0\,\bigr\},
\]
and it is nonempty. More precisely, set
\[
U_\varepsilon^+ := \Bigl\{\,v \in L^{p^\prime}(\mathbb{R}^N): \ \int_{\mathbb{R}^N}
\bigl(W_\varepsilon^{\,1/p} v\bigr)\,
\mathbf{R}\!\bigl(W_\varepsilon^{\,1/p}v\bigr)\,\mathrm{d}x \;>\; 0 \Bigr\}.
\]
For each $v \in U_\varepsilon^+$, there exists a unique $t_v>0$ such that $t_v v \in \mathcal{N}_\varepsilon$, and
\begin{equation}\label{2.12}
	t_v^{\,2-p^\prime}
	\;=\; \frac{\displaystyle \int_{\mathbb{R}^N} |v|^{p^\prime}\,\mathrm{d}x}
	{\displaystyle \int_{\mathbb{R}^N}
		\bigl(W_\varepsilon^{\,1/p} v\bigr)\,
		\mathbf{R}\!\bigl(W_\varepsilon^{\,1/p} v\bigr)\,\mathrm{d}x}.
\end{equation}
Furthermore, $t_v$ is the unique maximizer of the map $t\mapsto J_\varepsilon(tv)$, and
\[
c_\varepsilon := \inf_{\mathcal{N}_\varepsilon} J_\varepsilon
= \inf_{v\in U_\varepsilon^+} J_\varepsilon(t_v v) \;>\; 0.
\]
For all $v \in \mathcal{N}_\varepsilon$,
\[
c_\varepsilon \le J_\varepsilon(v)
= \Bigl(\frac{1}{p^\prime}-\frac{1}{2}\Bigr)\,\|v\|_{L^{p^\prime}(\mathbb{R}^N)}^{\,p^\prime}.
\]
Hence $0$ is an isolated point of the set $\{v \in L^{p^\prime}(\mathbb{R}^N): J_\varepsilon'(v)\,v=0\}$, and $\mathcal{N}_\varepsilon$ is a nonempty $C^1$ submanifold of $L^{p^\prime}(\mathbb{R}^N)$.

For later use, and following the approach of \cite[Lemma~2.1 and~2.2]{2017-Evequoz-AMPA}, we establish the two lemmas below and deduce that, when $W$ is  constant or $\mathbb{Z}^N$-periodic, the mountain-pass level is attained.

\begin{Lem}\label{lem:PS-basic}
	Let $(v_n)_n\subset L^{p^\prime}(\mathbb{R}^N)$ be a \((PS)\)-sequence for $J_\varepsilon$.
	Then $(v_n)_n$ is bounded, and there exists $v\in L^{p^\prime}(\mathbb{R}^N)$ with
	$J_\varepsilon'(v)=0$ such that, up to a subsequence,
	$v_n \rightharpoonup v$ weakly in $L^{p^\prime}(\mathbb{R}^N)$ and
	\[
	J_\varepsilon(v)\ \le\ \liminf_{n\to\infty} J_\varepsilon(v_n).
	\]
	Moreover, for every bounded measurable set $B\subset\mathbb{R}^N$,
	\[
	\mathbf{1}_B v_n \to \mathbf{1}_B v \quad \text{strongly in } L^{p^\prime}(\mathbb{R}^N).
	\]
\end{Lem}

\begin{Lem}\label{lem:MP}
	\begin{itemize}
		\item[(i)] The level $c_\varepsilon$ coincides with the mountain-pass value:
		\[
		c_\varepsilon \;=\; \inf_{\gamma\in \Gamma}\ \max_{t\in[0,1]} J_\varepsilon(\gamma(t)),
		\qquad
		 \Gamma:=\Bigl\{\gamma\in C([0,1],L^{p^\prime}(\mathbb{R}^N)):\ \gamma(0)=0,\ J_\varepsilon(\gamma(1))<0\Bigr\}.
		\]
		\item[(ii)] If $c_\varepsilon$ is achieved, then
		\[
		c_\varepsilon \;=\; \min\bigl\{\, J_\varepsilon(v):\ v\in L^{p^\prime}(\mathbb{R}^N)\setminus\{0\},\ J_\varepsilon'(v)=0 \,\bigr\}.
		\]
		\item[(iii)] If $W_\varepsilon$ is constant or $\mathbb{Z}^N$-periodic, then $c_\varepsilon$ is attained.
	\end{itemize}
\end{Lem}

Based on the above, if $v\in L^{p^\prime}(\mathbb{R}^N)\setminus\{0\}$ is a critical point at the mountain-pass level,
i.e.\ $J_\varepsilon'(v)=0$ and $J_\varepsilon(v)=c_\varepsilon$, then the function
\begin{equation}\label{eq:dual-u}
	u(x)\;=\; k^{\frac{4}{p-2}}\ \mathbf{R}\!\Big(W_\varepsilon^{\,1/p} v\Big)(k x),
	\qquad k=\varepsilon^{-1},
\end{equation}
is called a dual ground state solution of \eqref{2.1}.
More generally, if $v$ is any nontrivial critical point of $J_\varepsilon$, then the corresponding $u$ given by \eqref{eq:dual-u}
is called a dual solution of \eqref{2.1}.

\subsection{Representation lemma and the Palais--Smale condition}

We now study Palais--Smale sequences for $J_\varepsilon$ in more detail and first prove a representation lemma in the case where the coefficient $W$ is a positive constant. The key tool for the nonlocal quadratic term is the nonvanishing theorem in~\cite[Sec.~3]{2015-Evequoz&Weth-AIM}. For simplicity (and since the next result is independent of $\varepsilon$), we drop the subscript~$\varepsilon$.

\begin{Lem}\label{lem:profile}
	Assume that $W\equiv W(0)>0$ on $\mathbb{R}^N$.
	Let $(v_n)_n\subset L^{p^\prime}(\mathbb{R}^N)$ be a \((PS)_d\)-sequence for $J$ at level $d>0$.
	Then there exist an integer $m\ge1$, nontrivial critical points $w^{(1)},\dots,w^{(m)}$ of $J$,
	and sequences $(x_n^{(1)}),\dots,(x_n^{(m)})\subset\mathbb{R}^N$ such that ,  for a subsequence , as $n\to\infty$, we have
	\begin{equation}\label{eq:profile}
		\begin{cases}
			\ \ \ \|\,v_n - \sum_{j=1}^m w^{(j)}(\cdot - x_n^{(j)})\,\|_{L^{p^\prime}(\mathbb{R}^N)} \to 0, \\[2pt]
			\ |x_n^{(i)} - x_n^{(j)}| \to \infty \quad (i\neq j),\\[2pt]
			\ \ \ \sum_{j=1}^m J\!\big(w^{(j)}\big) = d.
		\end{cases}
	\end{equation}
\end{Lem}
\begin{proof}
	The proof follows \cite[Lemma~2.3]{2017-Evequoz-AMPA}; therefore, the details are omitted.
\end{proof}

We next address the Palais--Smale condition for $J_\varepsilon$.
If $W(x)\to 0$ as $|x|\to\infty$, then the condition holds at every energy level,
i.e.\ every Palais--Smale sequence admits a convergent subsequence (see \cite[Sec.~5]{2015-Evequoz&Weth-AIM}).

To deal with the case
\begin{equation}
	W_\infty \;:=\; \limsup_{|x|\to\infty} W(x) \;>\; 0,
\end{equation}
consider the limiting functional
\begin{equation}
	J_\infty(v)
	\;=\; \frac{1}{p^\prime} \int_{\mathbb{R}^N} |v|^{p^\prime}\,\mathrm{d}x
	\;-\; \frac{1}{2} \int_{\mathbb{R}^N} \bigl(W_\infty^{\,1/p} v\bigr)\,
	\mathbf{R}\!\bigl(W_\infty^{\,1/p} v\bigr)\,\mathrm{d}x,
	\qquad v\in L^{p^\prime}(\mathbb{R}^N).
\end{equation}
Its Nehari manifold is
\[
\mathcal{N}_\infty
:= \bigl\{\, v\in L^{p^\prime}(\mathbb{R}^N)\setminus\{0\}\ :\ J_\infty'(v)\,v=0 \,\bigr\}.
\]
Since $W_\infty$ is a constant, Lemma~\ref{lem:MP}(iii) yields that
\[
c_\infty \;:=\; \inf_{\mathcal{N}_\infty} J_\infty
\]
is achieved and coincides with the least energy among nontrivial critical points of $J_\infty$.

The final result of this section shows that the Palais--Smale condition holds for $J_\varepsilon$ on the Nehari manifold $\mathcal{N}_\varepsilon$ at all energy levels strictly below $c_\infty$. The proof is inspired by the arguments of Cingolani and Lazzo~\cite{1997-Cingolani&Lazzo-TMNA,2000-Cingolani&Lazzo-JDE}. The new difficulty lies in the nonlocal quadratic term of the energy functional, which produces a nonzero interaction between functions with disjoint supports. To overcome this, we first derive an estimate for this interaction in terms of the distance between the supports, based on the large-$|x|$ asymptotic expansion of the fundamental solution introduced in~\cite[Sec.~3]{2015-Evequoz&Weth-AIM}. With this estimate at hand, the Palais--Smale condition on $\mathcal{N}_\varepsilon$ for levels below $c_\infty$ follows.

\begin{Lem}\label{lem:offdiag}
	Assume that $(A_1)$ hold.
	There exists a constant $C=C(N,p)>0$ such that for any $R>0$, $r\ge 1$, and $u,v\in L^{p^\prime}(\mathbb{R}^N)$ with $\operatorname{supp}(u)\subset B_R$ and $\operatorname{supp}(v)\subset \mathbb{R}^N\setminus B_{R+r}$, one has
	\begin{equation}\label{eq:offdiag}
		\left|\int_{\mathbb{R}^N} u\, \mathscr{R}v\, \mathrm{d}x\right|
		\;\le\; C\, r^{-\lambda_p}\, \|u\|_{L^{p^\prime}(\mathbb{R}^N)}\, \|v\|_{L^{p^\prime}(\mathbb{R}^N)},
		\qquad
		\lambda_p=\frac{N-1}{2}-\frac{N+1}{p}.
	\end{equation}
\end{Lem}

\begin{proof}
	We give details for the case $a_1>0>a_2$ (the case $a_1>a_2>0$ requires only minor modifications).
	Let $G$ denote the fundamental solution of $L$ and let $\mathscr{R}$ be the analytic operator given (on Schwartz functions) by convolution with $G$. Since $R$ is the real part of $\mathscr{R}$ and $u,v$ are real-valued, it suffices to establish the estimate for Schwartz functions and then pass to the general case by density. Set $M_{R+r}:=\mathbb{R}^N\setminus B_{R+r}$,
	and take $u,v\in\mathcal{S}(\mathbb{R}^N)$ with $\operatorname{supp}(u)\subset B_R$ and $\operatorname{supp}(v)\subset M_{R+r}$.
	By the symmetry of the kernel of $\mathscr{R}$ and Hölder's inequality,
	\begin{equation}\label{fenjie}
		\left|\int_{\mathbb{R}^N} u\, \mathscr{R}v\, \mathrm{d}x\right|
		=\left|\int_{\mathbb{R}^N} v\, \mathscr{R}u\, \mathrm{d}x\right|
		\le \|v\|_{L^{p^\prime}(\mathbb{R}^N)}\, \|G*u\|_{L^p(M_{R+r})}.
	\end{equation}
	
	Choose $\psi\in\mathcal{S}(\mathbb{R}^N)$ with $\widehat{\psi}\in C_c^\infty(\mathbb{R}^N)$, $0\le \widehat{\psi}\le 1$, and
	\begin{equation*}
		\widehat{\psi}(\xi)=
		\begin{cases}
			1, & \big||\xi|-\sqrt{a_1}\big|\le \dfrac{\sqrt{a_1}}{6},\\[4pt]
			0, & \big||\xi|-\sqrt{a_1}\big|\ge \dfrac{\sqrt{a_1}}{4}.
		\end{cases}
	\end{equation*}
	Set
	\[
	G_1:=(2\pi)^{-N/2}\,(\psi*G),\qquad G_2:=G-G_1.
	\]
	From \eqref{guji} we have $|G(x)|\le C|x|^{\frac{1-N}{2}}$ for $|x|\ge 1$. Using $\psi\in\mathcal{S}$ and smoothing by convolution, it follows that
	\begin{equation}\label{eq:G1-decay}
		|G_1(x)|\le C(1+|x|)^{\frac{1-N}{2}}\quad \forall\,x\in\mathbb{R}^N.
	\end{equation}
	Again, Furthermore, thanks to \eqref{guji} and \eqref{eq:G1-decay}, we deduce
	\begin{equation}\label{g21}
		|G_2(x)|\le
		\begin{cases}
			C|x|^{4-N}, & N>4,\\[2pt]
			C\big(1+|\log|x||\big), & N=4,\\[2pt]
			C, & N\in\{2,3\}.
		\end{cases}
		\quad \text{ for } |x|\le 1
	\end{equation}
	Moreover, since $\widehat{G_2}=(1-\widehat{\psi})\widehat{G}$ and, when $a_2<0$, $\widehat{G}$ is given by \eqref{2.6}, one sees that for any multi-index $\gamma$ with $|\gamma|\ge N-3$ we have $\partial^\gamma \widehat{G_2}\in L^1(\mathbb{R}^N)$. Consequently, for all $x\in\mathbb{R}^N$,
	\begin{equation}\label{g22}
		|G_2(x)|\le
		\begin{cases}
			C\min\{|x|^{4-N},\,|x|^{-N}\}, & N>4,\\[2pt]
			C\min\{1+|\log|x||,\,|x|^{-N}\}, & N=4,\\[2pt]
			C\min\{1,\,|x|^{-N}\}, & N\in\{2,3\}.
		\end{cases}
	\end{equation}
	
	Since $\operatorname{supp}(u)\subset B_R$, we obtain
	\[
	\begin{aligned}
		\|G_2*u\|_{L^p(M_{R+r})}
		&\le \Bigg[\int_{|x|\ge R+r}\!\!\Big(\int_{|y|\le R}\! |G_2(x-y)|\,|u(y)|\,\mathrm{d}y\Big)^p \mathrm{d}x\Bigg]^{1/p}\\
		&\le \Bigg[\int_{\mathbb{R}^N}\!\Big(\int_{|x-y|\ge r}\! |G_2(x-y)|\,|u(y)|\,\mathrm{d}y\Big)^p \mathrm{d}x\Bigg]^{1/p}\\
		&= \big\|\big(\mathbf{1}_{M_r}|G_2|\big)*|u|\big\|_{L^p(\mathbb{R}^N)}
		\;\le\; \|\mathbf{1}_{M_r}G_2\|_{L^{p/2}(\mathbb{R}^N)}\,\|u\|_{L^{p^\prime}(\mathbb{R}^N)}.
	\end{aligned}
	\]
	Using \eqref{g22} for $|x|\ge r\ge 1$,
	\[
	\|\mathbf{1}_{M_r}G_2\|_{L^{p/2}(\mathbb{R}^N)}
	\le C\left(\int_{M_r} |x|^{- \frac{Np}{2}} \,\mathrm{d}x\right)^{\frac{2}{p}}
	\le C\, r^{-\frac{N(p-2)}{p}}
	\le C\, r^{-\lambda_p},
	\]
	whence
	\begin{equation}\label{G2}
		\|G_2*u\|_{L^p(M_{R+r})} \le C\, r^{-\lambda_p}\, \|u\|_{L^{p^\prime}(\mathbb{R}^N)}.
	\end{equation}
For $G_1$, fix a radial $\phi \in \mathcal{S}\left(\mathbb{R}^N\right)$ with radial $\widehat{\phi} \in C_c^{\infty}\left(\mathbb{R}^N\right)$, $0 \leq \widehat{\phi} \leq 1$, $\widehat{\phi}(\xi)=1$ when $\big||\xi|-\sqrt{a_1}\big| \leq \dfrac{\sqrt{a_1}}{4}$ and $\widehat{\phi}(\xi)=0$ when $\big||\xi|-\sqrt{a_1}\big| \geq \dfrac{\sqrt{a_1}}{2}$. Let $\bar{u}:=\phi * u \in \mathcal{S}\left(\mathbb{R}^N\right)$. By construction $\widehat{G_1}\,\widehat{\phi}=\widehat{G_1}$, hence
 Let $\tilde u:=\phi*u\in\mathcal{S}(\mathbb{R}^N)$. By construction $\widehat{G_1}\widehat{\phi}=\widehat{G_1}$, hence
	\[
	G_1*u=(2\pi)^{-N/2}\,(G_1*\tilde u).
	\]
	Decompose
	\[
	G_1*\tilde u=\big[\mathbf{1}_{B_{r/2}} G_1\big]*\tilde u+\big[\mathbf{1}_{M_{r/2}} G_1\big]*\tilde u,
	\]
	and set $g_r:=\big[\mathbf{1}_{B_{r/2}} G_1\big]*\phi$. Since $\operatorname{supp}(u)\subset B_R$, arguing as above yields
	\[
	\|\big[\mathbf{1}_{B_{r/2}} G_1\big]*\tilde u\|_{L^p(M_{R+r})}
	= \|g_r*u\|_{L^p(M_{R+r})}
	\le \|\mathbf{1}_{M_r} g_r\|_{L^{p/2}(\mathbb{R}^N)}\,\|u\|_{L^{p^\prime}(\mathbb{R}^N)}.
	\]
	Using \eqref{eq:G1-decay} and the rapid decay of $\phi$, we estimate for any large $m$,
	\[
	\begin{aligned}
		\|\mathbf{1}_{M_r} g_r\|_{L^{p/2}(\mathbb{R}^N)}^{p/2}
		&\le C_0^{p/2} \int_{|x|\ge r} \Big(\int_{|y|\le r/2} |\phi(x-y)|\,\mathrm{d}y\Big)^{p/2}\,\mathrm{d}x\\
		&\le C \int_{|x|\ge r} \Big(\int_{|y|\le r/2} |x-y|^{-m}\,\mathrm{d}y\Big)^{p/2}\,\mathrm{d}x\\
		&\le C\,|B_{r/2}|^{p/2}\int_{|x|\ge r}\big(|x|-r/2\big)^{-mp/2}\,\mathrm{d}x\\
		&= C\, r^{\frac{(N-m)p}{2}+N}\int_{|z|\ge 1}\big(|z|-1/2\big)^{-mp/2}\,\mathrm{d}z
		\;\le\; C\, r^{-\,\lambda_p},
	\end{aligned}
	\]
	provided $m$ is chosen so that $\tfrac{(m-N)p}{2}-N\ge \lambda_p$. Moreover, by~\cite[Prop.~3.3]{2015-Evequoz&Weth-AIM} we also have
	\[
	\big\|\big[\mathbf{1}_{M_{r/2}} G_1\big]*\tilde u\big\|_{L^p(M_{R+r})}
	\le \big\|\big[\mathbf{1}_{M_{r/2}} G_1\big]*\tilde u\big\|_{L^p(\mathbb{R}^N)}
	\le C\, r^{-\lambda_p}\,\|\tilde u\|_{L^{p^\prime}(\mathbb{R}^N)}
	\le C\, r^{-\lambda_p}\,\|u\|_{L^{p^\prime}(\mathbb{R}^N)}.
	\]
	and we conclude that
	\begin{equation}\label{G1}
		\|G_1*u\|_{L^p(M_{R+r})}\le C\, r^{-\lambda_p}\,\|u\|_{L^{p^\prime}(\mathbb{R}^N)}.
	\end{equation}
	Moreover, combining \eqref{fenjie}, \eqref{G2}, and \eqref{G1} yields \eqref{eq:offdiag}.
\end{proof}

\begin{Rem}
	Under assumption  $(A_2)$, Lemma~\ref{lem:offdiag} still holds. We only need to make minor modifications to the previous proof. The properties of $G_1$ are exactly the same as before; the difference is that $G_2$ satisfies \eqref{g21} but not \eqref{g22}. Note that $a_2=0$ implies that the Fourier symbol $\widehat{G}_2(\xi)$ behaves like $|\xi|^{-2}$ as $|\xi|\to 0$, and this low-frequency singularity is precisely what yields a different decay rate of $G_2$ at infinity. From \eqref{2.6} we obtain
	\[
	|G_2(x)|
	\le \frac{1}{\sqrt{\beta^{2}-4\alpha}}
	\Big( |\Phi_{a_1}(x)-(\psi*\Phi_{a_1})(x)| + |\Phi_0(x)| + |(\psi*\Phi_0)(x)| \Big)
	\le C\big(|H(x)|+|x|^{2-N}\big),
	\]
	where
	\[
	\widehat{H}:=(1-\widehat{\psi})\,\widehat{\Phi}_{a_1}.
	\]
	Since the prefactor $1-\widehat{\psi}$ removes the singularity of $\widehat{\Phi}_{a_1}$ and $\partial^{\gamma}\widehat{\Phi}_{a_1}(\xi)$ behaves like $|\xi|^{-2-|\gamma|}$ as $|\xi|\to\infty$, it follows that $\partial^{\gamma}\widehat{H}\in L^{1}(\mathbb{R}^N)$ for all multi-indices $\gamma$ with $|\gamma|\ge N-1$. In particular, when $|x|\ge 1$,
	\[
	|H(x)|\le C\,|x|^{2-N}.
	\]
	Combining this with \eqref{g21} yields
	\[
	|G_2(x)|\le
	\begin{cases}
		C\,\min\{|x|^{4-N},\,|x|^{2-N}\}, & N>4,\\[2mm]
		C\,\min\{1+|\log|x||,\,|x|^{2-N}\}, & N=4,\\[2mm]
		C\,\min\{1,\,|x|^{2-N}\}, & N=3.
	\end{cases}
	\]
	Hence,
	\[
	\| \mathbf{1}_{M_r}\, G_2 \|_{L^{p/2}}^{\,p/2}
	\;\le\; C \int_{M_r} |x|^{-\frac{(N-2)p}{2}} \,\mathrm{d}x
	\;=\; C\,\omega_N \int_{r}^{\infty} s^{\,N-1-\frac{(N-2)p}{2}} \,\mathrm{d}s,
	\]
	and therefore
	\[
	\| \mathbf{1}_{M_r}\, G_2 \|_{L^{p/2}}
	\;\le\; C \left(\omega_N \int_{r}^{\infty} s^{\,N-1-\frac{(N-2)p}{2}} \,\mathrm{d}s\right)^{\!2/p}
	\;\le\; C\, r^{-\frac{(N-2)(p-2)}{p}}
	\;\le\; C\, r^{-\lambda_p}.
	\]
	Since $r\ge 1$, we obtain
	\[
	\|G_2 * u\|_{L^p(M_{R+r})} \;\le\; C\, r^{-\lambda_p}\, \|u\|_{L^{p^\prime}(\mathbb{R}^N)}.
	\]
	
	Finally, combining this with the estimate for $G_1$ in \eqref{G1}, we conclude that Lemma~\ref{lem:offdiag} also holds under $(A_2)$.
\end{Rem}

\begin{Lem}\label{lem:PS-below-cinf}
	Let $\varepsilon>0$ and assume $W_{\infty}>0$ and $c_{\varepsilon}<c_{\infty}$.
	Then $J_{\varepsilon}$ satisfies the Palais--Smale condition below $c_\infty$; that is,
	if $(v_n)_n\subset \mathcal{N}_{\varepsilon}$ with $J_{\varepsilon}(v_n)\to d<c_{\infty}$ and
	$(J_{\varepsilon}|_{\mathcal{N}_{\varepsilon}})'(v_n)\to 0$, then $(v_n)_n$ admits a convergent
	subsequence in $L^{p^\prime}(\mathbb{R}^N)$.
\end{Lem}

\begin{proof}
	The proof follows \cite[Lemma~2.5]{2017-Evequoz-AMPA}; therefore, the details are omitted.
\end{proof}

\begin{Rem}
	Under the stronger assumption $W_\infty=\lim\limits_{|x|\to\infty}W(x)$,
	the proof simplifies: after extracting a weakly convergent subsequence with limit $v$ (a critical
	point of $J_\varepsilon$), the sequence $w_n:=v_n-v$ is a \((PS)\)-sequence for $J_\infty$
	at a level strictly below $c_\infty$. Applying the representation lemma (Lemma~\ref{lem:profile}) then yields
	$w_n\to 0$ in $L^{p^\prime}(\mathbb{R}^N)$.
\end{Rem}

\section{Existence and concentration of dual ground states}

In this section and the next, we assume that $W(x)$ satisfies conditions $(W_1)$ and $(W_2)$,  consider the functional
\[
J_0(v)
:= \frac{1}{p^\prime}\int_{\mathbb{R}^N} |v|^{p^\prime}\,\mathrm{d}x
- \frac{1}{2}\int_{\mathbb{R}^N} \big(W_0^{1/p} v\big)\,
\mathbf{R}\big(W_0^{1/p} v\big)\,\mathrm{d}x,
\qquad v\in L^{p^\prime}(\mathbb{R}^N),
\]
and the associated Nehari manifold
\[
\mathcal{N}_0
:= \bigl\{\, v\in L^{p^\prime}(\mathbb{R}^N)\setminus\{0\}:\ J_0'(v)\,v=0 \,\bigr\}.
\]
This is linked to the limit problem
\begin{equation}\label{eq:limit}
	\Delta^2 u - \beta \Delta u + \alpha u
	= W_0\,|u|^{p-2}u,
	\qquad x\in\mathbb{R}^N.
\end{equation}

By Lemma~\ref{lem:MP}\textup{(iii)}, the level $c_0:=\inf\limits_{\mathcal{N}_0}J_0$ is attained and coincides with the least
energy among nontrivial critical points:
\[
c_0 = \inf\bigl\{\, J_0(v):\ v\in L^{p^\prime}(\mathbb{R}^N),\ v\neq0,\ J_0'(v)=0 \,\bigr\}.
\]

Let
\[
M:=\bigl\{\,x\in\mathbb{R}^N:\ W(x)=W_0 \,\bigr\}.
\]
By ($W_1$)--($W_2$), $M\neq\varnothing$ and $M$ is compact. We construct test functions by translating and cutting off a ground state of $J_0$, then projecting onto $\mathcal{N}_\varepsilon$.
Fix $\eta\in C_c^\infty(\mathbb{R}^N)$ with $0\le \eta\le 1$, $\eta\equiv1$ on $B_1(0)$, and
$\eta\equiv0$ on $\mathbb{R}^N\setminus B_2(0)$.
For $y\in M$ and $\varepsilon>0$ set
\begin{equation}\label{eq:phi}
	\varphi_{\varepsilon,y}(x):=\eta(\varepsilon x - y)\, w\!\bigl(x-\varepsilon^{-1}y\bigr),
\end{equation}
where $w\in L^{p^\prime}(\mathbb{R}^N)$ is a fixed least-energy critical point of $J_0$.

\begin{Lem}\label{lem:31}
	There exists $\varepsilon^*>0$ such that, for all $0<\varepsilon\le\varepsilon^*$ and $y\in M$, there
	is a unique $t_{\varepsilon,y}>0$ with $t_{\varepsilon,y}\varphi_{\varepsilon,y}\in \mathcal{N}_\varepsilon$.
	Moreover,
	\[
	\lim_{\varepsilon\to0^+} J_\varepsilon\bigl(t_{\varepsilon,y}\varphi_{\varepsilon,y}\bigr)=c_0,
	\quad\text{uniformly in } y\in M.
	\]
\end{Lem}

\begin{proof}
	As $\varepsilon\to0^+$,
	\[
	W(y+\varepsilon\cdot)\,\eta(\varepsilon\cdot)\,w \;\to\; W_0\,w
	\quad\text{in } L^{p^\prime}(\mathbb{R}^N),
	\]
	uniformly in $y\in M$ by continuity of $W$ and compactness of $M$.
	With the change of variables $x=z+\varepsilon^{-1}y$,
	\[
	\begin{aligned}
		&\quad	\int_{\mathbb{R}^N}\!\big(W_\varepsilon^{1/p}\varphi_{\varepsilon,y}\big)\,
		\mathbf{R}\big(W_\varepsilon^{1/p}\varphi_{\varepsilon,y}\big)\,\mathrm{d}x\\
		&=\int_{\mathbb{R}^N}\!W^{1/p}(y+\varepsilon z)\,\eta(\varepsilon z)w(z)\,
		\mathbf{R}\big(W^{1/p}(y+\varepsilon z)\,\eta(\varepsilon z)w\big)(z)\,\mathrm{d}z\\	&\longrightarrow \int_{\mathbb{R}^N}\!\big(W_0^{1/p}w\big)\,\mathbf{R}\big(W_0^{1/p}w\big)\,\mathrm{d}z
		=\Bigl(\frac{1}{p^\prime}-\frac{1}{2}\Bigr)^{-1} c_0 >0,
	\end{aligned}
	\]
	uniformly in $y\in M$. Hence $\varphi_{\varepsilon,y}\in U_\varepsilon^+$ for small $\varepsilon$, and $t_{\varepsilon,y}$
	exists by \eqref{2.12}. Also,
	\[
	\int_{\mathbb{R}^N}|\varphi_{\varepsilon,y}|^{p^\prime}\,\mathrm{d}x
	=\int_{\mathbb{R}^N}|\eta(\varepsilon z)w(z)|^{p^\prime}\,\mathrm{d}z
	\longrightarrow \int_{\mathbb{R}^N}|w|^{p^\prime}\,\mathrm{d}z
	=\Bigl(\frac{1}{p^\prime}-\frac{1}{2}\Bigr)^{-1}c_0,
	\]
	so $t_{\varepsilon,y}\to1$ uniformly in $y$, and
	$J_\varepsilon(t_{\varepsilon,y}\varphi_{\varepsilon,y})\to c_0$.
\end{proof}

\begin{Lem}\label{lem:32}
	For all $\varepsilon>0$, one has $c_\varepsilon\ge c_0$. Moreover,
	\[
	\lim_{\varepsilon\to0^+} c_\varepsilon = c_0.
	\]
\end{Lem}

\begin{proof}
	Let $v_\varepsilon\in\mathcal{N}_\varepsilon$ be arbitrary and set
	\[
	v_0:=\Bigl(\frac{W_\varepsilon}{W_0}\Bigr)^{\!1/p}\,v_\varepsilon.
	\]
	Then $W_0^{1/p}v_0=W_\varepsilon^{1/p}v_\varepsilon$, hence
	\[
	\int_{\mathbb{R}^N}\big(W_0^{1/p}v_0\big)\,\mathbf{R}\big(W_0^{1/p}v_0\big)\,\mathrm{d}x
	=\int_{\mathbb{R}^N}\big(W_\varepsilon^{1/p}v_\varepsilon\big)\,\mathbf{R}\big(W_\varepsilon^{1/p}v_\varepsilon\big)\,\mathrm{d}x
	=\int_{\mathbb{R}^N}|v_\varepsilon|^{p^\prime}\,\mathrm{d}x,
	\]
	since $v_\varepsilon\in\mathcal{N}_\varepsilon$.
	Moreover,
	\[
	\int_{\mathbb{R}^N}|v_0|^{p^\prime}\,\mathrm{d}x
	=\int_{\mathbb{R}^N}\Bigl(\frac{W_\varepsilon}{W_0}\Bigr)^{\!\frac{p^\prime}{p}}|v_\varepsilon|^{p^\prime}\,\mathrm{d}x
	\le \int_{\mathbb{R}^N}|v_\varepsilon|^{p^\prime}\,\mathrm{d}x,
	\]
	because $W_\varepsilon\le W_0$.
	Let $t_\varepsilon>0$ be such that $t_\varepsilon v_0\in\mathcal{N}_0$. Then
	\[
	t_\varepsilon^{\,2-p^\prime}=\frac{\displaystyle\int|v_0|^{p^\prime}\,\mathrm{d}x}
	{\displaystyle\int (W_0^{1/p}v_0)\,\mathbf{R}(W_0^{1/p}v_0)\,\mathrm{d}x}
	\le 1,
	\]
	so
	\[
	c_0 \le J_0(t_\varepsilon v_0)
	=\Bigl(\frac{1}{p^\prime}-\frac{1}{2}\Bigr)t_\varepsilon^{p^\prime}\!\int|v_0|^{p^\prime}\,\mathrm{d}x
	\le \Bigl(\frac{1}{p^\prime}-\frac{1}{2}\Bigr)\!\int|v_\varepsilon|^{p^\prime}\,\mathrm{d}x
	= J_\varepsilon(v_\varepsilon).
	\]
	Taking $\inf$ over $v_\varepsilon\in\mathcal{N}_\varepsilon$ gives $c_0\le c_\varepsilon$.
	On the other hand, by Lemma~\ref{lem:31}, for any $y\in M$,
	\[
	c_\varepsilon \le J_\varepsilon\bigl(t_{\varepsilon,y}\varphi_{\varepsilon,y}\bigr)\to c_0
	\quad (\varepsilon\to0^+),
	\]
	hence $\lim\limits_{\varepsilon\to0^+} c_\varepsilon=c_0$.
\end{proof}

\begin{Prop}\label{prop:33} Assume $(A_1)$ or $(A_2)$.
	There exists $\varepsilon_0>0$ such that for all $0<\varepsilon<\varepsilon_0$, the least energy level $c_\varepsilon$ is attained.
\end{Prop}

\begin{proof}
	By Lemma~\ref{lem:32} and ($W_2$), there is $\varepsilon_0>0$ with $c_\varepsilon<c_\infty$ for all $0<\varepsilon<\varepsilon_0$.
	For such $\varepsilon$, since $\mathcal{N}_\varepsilon$ is a $C^1$ submanifold of $L^{p^\prime}(\mathbb{R}^N)$, Ekeland’s
	variational principle (see, e.g., \cite[Thm.~3.1]{1974-Ekeland-JMAA}) yields a \((PS)\)-sequence for $J_\varepsilon$ on
	$\mathcal{N}_\varepsilon$ at level $c_\varepsilon$. The conclusion follows from Lemma~\ref{lem:PS-below-cinf}.
\end{proof}

\begin{Prop}\label{prop:34}
	Let $(\varepsilon_n)_n\subset(0,\infty)$ with $\varepsilon_n\to 0$.
	For each $n$, take $v_n\in\mathcal{N}_{\varepsilon_n}$ such that $J_{\varepsilon_n}(v_n)\to c_0$.
	Then there exist $x_0\in M:=\{x:W(x)=W_0\}$, a critical point $w_0$ of $J_0$ with
	$J_0(w_0)=c_0$, and a sequence $(y_n)_n\subset\mathbb{R}^N$ such that, up to a subsequence,
	\[
	x_n:=\varepsilon_n y_n\to x_0,\qquad
	v_n(\cdot+y_n)\ \to\ w_0 \quad\text{in } L^{p^\prime}(\mathbb{R}^N).
	\]
\end{Prop}

\begin{proof}
	For each $n$, set
	\[
	v_{0,n}:=\Bigl(\frac{W_{\varepsilon_n}}{W_0}\Bigr)^{\!1/p}\,v_n.
	\]
	Then $0\le |v_{0,n}|\le |v_n|$ a.e. and
	\[
	\int_{\mathbb{R}^N}\big(W_0^{1/p} v_{0,n}\big)\,
	\mathbf{R}\big(W_0^{1/p} v_{0,n}\big)\,\mathrm{d}x
	=\int_{\mathbb{R}^N}\big(W_{\varepsilon_n}^{1/p} v_n\big)\,
	\mathbf{R}\big(W_{\varepsilon_n}^{1/p} v_n\big)\,\mathrm{d}x>0.
	\]
	Let $t_{0,n}>0$ be such that $t_{0,n}v_{0,n}\in\mathcal{N}_0$. Then
	\[
	t_{0,n}^{\,2-p^\prime}
	=\frac{\displaystyle\int|v_{0,n}|^{p^\prime}\,\mathrm{d}x}
	{\displaystyle\int (W_0^{1/p}v_{0,n})\,\mathbf{R}(W_0^{1/p}v_{0,n})\,\mathrm{d}x}
	=\frac{\displaystyle\int \Bigl(\frac{W_{\varepsilon_n}}{W_0}\Bigr)^{\!\frac{p^\prime}{p}} |v_n|^{p^\prime}\,\mathrm{d}x}
	{\displaystyle\int|v_n|^{p^\prime}\,\mathrm{d}x}\ \le\ 1,
	\]
	because $W_{\varepsilon_n}\le W_0$.
	Hence $t_{0,n}v_{0,n}\in\mathcal{N}_0$ and
%
	\[
	\begin{aligned}
		c_0 & \le J_0\!\left(t_{0,n}v_{0,n}\right)
		=\left(\frac{1}{p'}-\frac{1}{2}\right)t_{0,n}^{2}
		\int_{\mathbb{R}^N} W_0^{\frac{1}{p}} v_{0,n}\,
		\mathbf{R}\!\left(W_0^{\frac{1}{p}} v_{0,n}\right)\,\mathrm{d}x \\
		&=\left(\frac{1}{p'}-\frac{1}{2}\right)t_{0,n}^{2}
		\int_{\mathbb{R}^N} W_{\varepsilon_n}^{\frac{1}{p}} v_n\,
		\mathbf{R}\!\left(W_{\varepsilon_n}^{\frac{1}{p}} v_n\right)\,\mathrm{d}x \\
		&=t_{0,n}^{2}J_{\varepsilon_n}(v_n)
		\le J_{\varepsilon_n}(v_n)\to c_0,\qquad \text{as }n\to\infty.
	\end{aligned}
	\]
	
	Thus $t_{0,n}\to 1$, and $(t_{0,n}v_{0,n})_n\subset\mathcal{N}_0$ is a minimizing sequence for $J_0$ on $\mathcal{N}_0$.
	By Ekeland’s variational principle and the natural constraint property of $\mathcal{N}_0$,
	there exists a \((PS)_{c_0}\)-sequence $(w_n)_n\subset L^{p^\prime}(\mathbb{R}^N)$ for $J_0$ such that
	\[
	\|t_{0,n}v_{0,n}-w_n\|_{L^{p^\prime}(\mathbb{R}^N)+}\to 0.
	\]
	Applying Lemma~\ref{lem:profile}, there are a critical point $w_0$ of $J_0$ at level $c_0$ and
	$(y_n)_n\subset\mathbb{R}^N$ such that, up to a subsequence,
	\[
	\|\,w_n(\cdot+y_n)-w_0\,\|_{L^{p^\prime}(\mathbb{R}^N)}\to 0.
	\]
	Since $t_{0,n}\to 1$, it follows that
	\[
	v_{0,n}(\cdot+y_n)\to w_0 \quad \text{in } L^{p^\prime}(\mathbb{R}^N).
	\]
	
	We claim $(\varepsilon_n y_n)_n$ is bounded. Suppose not: assume $|\varepsilon_n y_n|\to\infty$
	along a subsequence. Using
	\[
	|v_n(x)|^{p^\prime}=\Bigl(\frac{W_0}{W(\varepsilon_nx+\varepsilon_n y_n)}\Bigr)^{p^\prime-1}
	|v_{0,n}(x+y_n)|^{p^\prime}
	\]
	and Fatou’s lemma,
	\[
	\begin{aligned}
		c_0
		&=\lim_{n\to\infty} J_{\varepsilon_n}(v_n)
		=\lim_{n\to\infty}\Bigl(\tfrac1{p^\prime}-\tfrac12\Bigr)\int |v_n|^{p^\prime}\,\mathrm{d}x \\
		&=\liminf_{n\to\infty}\Bigl(\tfrac1{p^\prime}-\tfrac12\Bigr)
		\int \Bigl(\tfrac{W_0}{W(\varepsilon_nx+\varepsilon_n y_n)}\Bigr)^{p^\prime-1}
		|v_{0,n}(x+y_n)|^{p^\prime}\,\mathrm{d}x \\
		&\ge \Bigl(\tfrac{W_0}{W_\infty}\Bigr)^{p^\prime-1}
		\Bigl(\tfrac1{p^\prime}-\tfrac12\Bigr)\int |w_0|^{p^\prime}\,\mathrm{d}x
		=\Bigl(\tfrac{W_0}{W_\infty}\Bigr)^{p^\prime-1} c_0,
	\end{aligned}
	\]
	which contradicts ($W_2$). Hence $(\varepsilon_n y_n)$ is bounded; passing to a subsequence,
	$\varepsilon_n y_n\to x_0\in\mathbb{R}^N$. Since $W(\varepsilon_nx+\varepsilon_n y_n)\to W(x_0)$
	locally uniformly, dominated convergence gives
	\[
	c_0=\Bigl(\tfrac{W_0}{W(x_0)}\Bigr)^{p^\prime-1} c_0,
	\]
	so $W(x_0)=W_0$, i.e.\ $x_0\in M$. Finally,
	\[
	v_n(\cdot+y_n)
	=\Bigl(\tfrac{W_0}{W(\varepsilon_n\cdot+\varepsilon_n y_n)}\Bigr)^{\!1/p} v_{0,n}(\cdot+y_n)
	\to \Bigl(\tfrac{W_0}{W(x_0)}\Bigr)^{\!1/p} w_0=w_0
	\]
	in $L^{p^\prime}(\mathbb{R}^N)$. This completes the proof.
\end{proof}

\begin{Thm}\label{thm:35}
	Assume $(A_1)$ or $(A_2)$.
	Let $k_0:=\varepsilon_0^{-1}>0$, where $\varepsilon_0>0$ is given by Proposition~\ref{prop:33}.
	Take any sequence $(k_n)_n\subset(k_0,\infty)$ with $k_n\to\infty$, and for each $n$ a dual ground state
	solution $u_n$ of
	\[
	\Delta^2 u - \beta k_n^2\,\Delta u + \alpha k_n^4\,u
	= W(x)\,|u|^{p-2}u \quad \text{in }\mathbb{R}^N.
	\]
	Then there exist $x_0\in M$, a dual ground state solution $u_0$ of the
	limit equation \eqref{eq:limit}, and a sequence $(x_n)_n\subset\mathbb{R}^N$ such that, up to a subsequence
	\[
	x_n\to x_0 \quad\text{and}\quad
	k_n^{-\frac{4}{p-2}}\, u_n\!\Big(\tfrac{\cdot}{k_n}+x_n\Big)\ \to\ u_0
	\quad \text{in } L^p(\mathbb{R}^N).
	\]
\end{Thm}

\begin{proof}
	Let $\varepsilon_n:=k_n^{-1}$. By the dual framework, there exist least-energy critical points
	$v_n\in L^{p^\prime}(\mathbb{R}^N)$ of $J_{\varepsilon_n}$ such that
	\[
	u_n(x)=k_n^{\frac{4}{p-2}}\,
	\mathbf{R}\!\Big(W_{\varepsilon_n}^{1/p} v_n\Big)\!\big(k_n x\big),
	\qquad W_{\varepsilon_n}(x):=W(\varepsilon_n x),
	\]
	and $\mathbf{R}:L^{p^\prime}\to L^p$ is continuous. By Lemma~\ref{lem:32} and Proposition~\ref{prop:34},
	there exist $x_0\in M$ and translations $y_n$ such that (up to a subsequence)
	\[
	x_n:=\varepsilon_n y_n\to x_0,\qquad v_n(\cdot+y_n)\to w_0 \ \text{ in } L^{p^\prime}(\mathbb{R}^N),
	\]
	where $w_0$ is a least-energy critical point of $J_0$.
	For any $x$,
	\[
	k_n^{-\frac{4}{p-2}}\, u_n\!\Big(\tfrac{x}{k_n}+x_n\Big)
	= \mathbf{R}\!\Big(W_{\varepsilon_n}^{1/p} v_n\Big)(x+y_n)
	= \mathbf{R}\!\Big(W_{\varepsilon_n}^{1/p}(\cdot+y_n)\, v_n(\cdot+y_n)\Big)(x).
	\]
	Since $W_{\varepsilon_n}(\cdot+y_n)=W(x_n+\varepsilon_n\cdot)\to W(x_0)=W_0$ pointwise
	and $\mathbf{R}:L^{p^\prime}\to L^p$ is continuous, we obtain
	\[
	k_n^{-\frac{4}{p-2}}\, u_n\!\Big(\tfrac{\cdot}{k_n}+x_n\Big)
	\ \to\ \mathbf{R}\!\big(W_0^{1/p} w_0\big)=:u_0
	\quad \text{in } L^p(\mathbb{R}^N),
	\]
	and $u_0$ solves \eqref{eq:limit}.
\end{proof}

\begin{Rem}
	(i) The conclusion of Theorem~\ref{thm:35} also holds for any sequence of dual bound states:
	it suffices to assume $v_n$ are critical points of $J_{\varepsilon_n}$ with $J_{\varepsilon_n}(v_n)\to c_0$.
	
	(ii) By \cite[Propositions~5.1 and 5.2]{2019-Denis-Jean-Rainer-JLMS}, the convergence to $u_0$ holds in $W^{4,q}(\mathbb{R}^N)$ for all $1<q<\infty$.
	In particular, taking $q>\frac{N}{4}$ yields convergence in $L^\infty(\mathbb{R}^N)$ as well.
	Moreover, since $u_0\in W^{4,q}(\mathbb{R}^N)$, for any $\delta>0$ there exists $R_\delta>0$ such that, for all sufficiently large $n$,
	\[
	k_n^{-\frac{4}{p-2}}\;|u_n(x)|\;<\;\delta
	\quad \text{for all }\ |x-x_n|\ \ge\ \frac{R_\delta}{k_n},
	\]
	while at the same time
	\[
	k_n^{-\frac{4}{p-2}}\;\|u_n\|_{L^\infty(\mathbb{R}^N)}
	\;\longrightarrow\;
	\|u_0\|_{L^\infty(\mathbb{R}^N)}\;>\;0
	\quad \text{as } n\to\infty .
	\]
	In addition, if $\tilde{x}_n$ denotes any global maximizer of $|u_n|$, then $\tilde{x}_n\to x_0$ as $n\to\infty$.
\end{Rem}

\section{Multiplicity and concentration of solutions}

To obtain the multiplicity result, fix $\delta>0$ and choose $\rho>0$ such that $M_\delta\subset B_\rho(0)$. Define the truncation map
$$
\Psi:\mathbb{R}^N\to\mathbb{R}^N,\qquad
\Psi(x)=
\begin{cases}
	x, & |x|<\rho,\\[1mm]
	\dfrac{\rho x}{|x|}, & |x|\ge\rho.
\end{cases}
$$
For each $\varepsilon>0$, define $\beta_\varepsilon: L^{p'}(\mathbb{R}^N)\setminus\{0\}
\to\mathbb{R}^N$ as
$$
\beta_\varepsilon(v)
:=\frac{1}{\|v\|_{L^{p'}(\mathbb{R}^N)}^{p'}}
\int_{\mathbb{R}^N}\Psi(\varepsilon x)\,|v(x)|^{p'}\,\mathrm{d}x.
$$
Then
\begin{equation}\label{4-1}
	\lim_{\varepsilon\to0^+}\beta_\varepsilon(\varphi_{\varepsilon,y})
	=y,~\text{uniformly in }y\in M_\delta.
\end{equation}

Indeed, since
$$
\begin{aligned}
	\beta_\varepsilon(\varphi_{\varepsilon,y})
	&=
	\frac{\displaystyle\int_{\mathbb{R}^N}
		\Psi(\varepsilon x)\,|\varphi_{\varepsilon,y}(x)|^{p'}\,\mathrm{d}x}
	{\displaystyle\int_{\mathbb{R}^N}|\varphi_{\varepsilon,y}(x)|^{p'}\,\mathrm{d}x}\\
	&=
	\frac{\displaystyle\int_{\mathbb{R}^N}
		\Psi(\varepsilon x)\,\eta(\varepsilon x-y)^{p'}\,
		\bigl|w(x-\varepsilon^{-1}y)\bigr|^{p'}\,\mathrm{d}x}
	{\displaystyle\int_{\mathbb{R}^N}
		\eta(\varepsilon x-y)^{p'}\,
		\bigl|w(x-\varepsilon^{-1}y)\bigr|^{p'}\,\mathrm{d}x}\\
	&=\frac{\displaystyle
		\int_{\mathbb{R}^N}
		\Psi\bigl(\varepsilon z+y\bigr)\,
		\eta(\varepsilon z)^{p'}|w(z)|^{p'}\,\mathrm{d}z}
	{\displaystyle\int_{\mathbb{R}^N}
		\eta(\varepsilon z)^{p'}|w(z)|^{p'}\,\mathrm{d}z},
\end{aligned}
$$
it follows from the Lebesgue dominated convergence theorem that
$$
\int_{\mathbb{R}^N}
\Psi(\varepsilon z+y)\,\eta(\varepsilon z)^{p'}|w(z)|^{p'}\,\mathrm{d}z
\to
y\int_{\mathbb{R}^N}|w(z)|^{p'}\,\mathrm{d}z,
$$
and
$$
\int_{\mathbb{R}^N}\eta(\varepsilon z)^{p'}|w(z)|^{p'}\,\mathrm{d}z
\to
\int_{\mathbb{R}^N}|w(z)|^{p'}\,\mathrm{d}z.
$$
Therefore, for each fixed $y\in M_\delta$, as $\varepsilon\to 0^+$, we have $\beta_\varepsilon(\varphi_{\varepsilon,y})\to y$.

Moreover, this convergence is uniform in $y\in M_\delta$. In fact, $M_\delta$ is bounded and closed, hence compact, and $\Psi$ is uniformly continuous on bounded subsets of $\mathbb{R}^N$. Consequently,
$$
\sup_{y\in M_\delta}\bigl|\beta_\varepsilon(\varphi_{\varepsilon,y})-y\bigr|\to 0
~\text{as}~\varepsilon\to 0^+,
$$
which proves \eqref{4-1}.

Next, we introduce the sublevel set
$$
\Sigma_\varepsilon
:=\{v\in\mathcal{N}_\varepsilon:\ J_\varepsilon(v)\le c_0+\nu(\varepsilon)\},
$$
where $\nu:(0,\infty)\to(0,\infty)$ satisfies
$$
\nu(\varepsilon)>c_\varepsilon-c_0,\qquad
\nu(\varepsilon)\to0\quad(\varepsilon\to0^+).
$$
First, since $c_\varepsilon<c_0+\nu(\varepsilon)$ and $c_\varepsilon=\inf_{\mathcal{N}_\varepsilon}J_\varepsilon$, it follows that $\Sigma_\varepsilon\neq\varnothing$.

Let $(\varepsilon_n)$ be an arbitrary sequence with $\varepsilon_n\to0^+$. For each $n$, choose $v_n\in\Sigma_{\varepsilon_n}$ such that
$$
\inf_{y\in M_{\delta/2}}\bigl|\beta_{\varepsilon_n}(v_n)-y\bigr|
\ \ge\
\sup_{v\in\Sigma_{\varepsilon_n}}
\inf_{y\in M_{\delta/2}}\bigl|\beta_{\varepsilon_n}(v)-y\bigr|
-\frac1n.
$$
Since
$$
c_{\varepsilon_n}\le J_{\varepsilon_n}(v_n)\le c_0+\nu(\varepsilon_n),
$$
and $c_{\varepsilon_n}\to c_0$, $\nu(\varepsilon_n)\to0$, we deduce that
$$
J_{\varepsilon_n}(v_n)\to c_0.
$$

By Proposition~\ref{prop:34}, there exist $x_0\in M$, a least energy critical point $w_0$ of the limit functional $J_0$, and a sequence $(y_n)\subset\mathbb{R}^N$ such that, up to a subsequence,
$$
x_n:=\varepsilon_n y_n\to x_0,\qquad
v_n(\cdot+y_n)\to w_0\quad\text{in }L^{p'}(\mathbb{R}^N).
$$
From the definition of $\beta_{\varepsilon_n}(v_n)$ we have
$$
\begin{aligned}
	\beta_{\varepsilon_n}(v_n)
	=\frac{\displaystyle\int_{\mathbb{R}^N}
		\Psi\bigl(\varepsilon_n z+x_n\bigr)\,
		|w_n(z)|^{p'}\,\mathrm{d}z}
	{\displaystyle\int_{\mathbb{R}^N}|w_n(z)|^{p'}\,\mathrm{d}z},
\end{aligned}
$$
where $w_n(z):=v_n(z+y_n)$. By the definition of $\Psi$ and its continuity on $\mathbb{R}^N$ we have
$$
\Psi\bigl(\varepsilon_n z+x_n\bigr)\to \Psi(x_0)=x_0
\quad\text{for each fixed }z\in\mathbb{R}^N,
$$
and, for all $n$ and $z$,
$$
\bigl|\Psi(\varepsilon_n z+x_n)\bigr|\le\rho.
$$

Since $w_n\to w_0$ in $L^{p'}(\mathbb{R}^N)$, the continuity of the Nemytskii operator $u\mapsto |u|^{p'}$ implies that $|w_n|^{p'}\to|w_0|^{p'}$ in $L^1(\mathbb{R}^N)$. 
Using the uniform bound $|\Psi(\varepsilon_n z+x_n)|\le\rho$, by the Lebesgue dominated convergence theorem, we obtain
$$
\int_{\mathbb{R}^N}
\Psi\bigl(\varepsilon_n z+x_n\bigr)\,
|w_n(z)|^{p'}\,\mathrm{d}z
\ \to\
x_0\int_{\mathbb{R}^N}|w_0(z)|^{p'}\,\mathrm{d}z.
$$

Similarly, since $|w_n|^{p'}\to|w_0|^{p'}$ in $L^1(\mathbb{R}^N)$,
$$
\int_{\mathbb{R}^N}|w_n(z)|^{p'}\,\mathrm{d}z
\to
\int_{\mathbb{R}^N}|w_0(z)|^{p'}\,\mathrm{d}z.
$$
Therefore,
$$
\beta_{\varepsilon_n}(v_n)\to x_0\in M.
$$

In particular, for all sufficiently large $n$ we have
$$
\operatorname{dist}\bigl(\beta_{\varepsilon_n}(v_n),M_{\delta/2}\bigr)
=\inf_{y\in M_{\delta/2}}|\beta_{\varepsilon_n}(v_n)-y|
<\frac{\delta}{2}.
$$
By the choice of $v_n$, this implies that
$$
\sup_{v\in\Sigma_{\varepsilon_n}}
\inf_{y\in M_{\delta/2}}|\beta_{\varepsilon_n}(v)-y|
\le \inf_{y\in M_{\delta/2}}|\beta_{\varepsilon_n}(v_n)-y|+\frac1n
\to 0.
$$
Since the sequence $(\varepsilon_n)$ is arbitrary, we arrive at
\begin{equation}\label{4-2}
	\lim_{\varepsilon\to0^+}\
	\sup_{v\in\Sigma_\varepsilon}\
	\inf_{y\in M_{\delta/2}}\bigl|\beta_\varepsilon(v)-y\bigr|
	=0.
\end{equation}

\medskip

\noindent\textbf{Proof of Theorem~1.2.}
\emph{Multiplicity of solutions.} For any $y\in M$ and $\varepsilon>0$, consider the test function $\varphi_{\varepsilon,y}$ constructed earlier, and let
$$
t_{\varepsilon,y}>0
$$
be the unique real number such that
$$
t_{\varepsilon,y}\varphi_{\varepsilon,y}\in\mathcal{N}_\varepsilon,
$$
equivalently,
$$
J_\varepsilon\bigl(t_{\varepsilon,y}\varphi_{\varepsilon,y}\bigr)
=\max_{t>0}J_\varepsilon(t\varphi_{\varepsilon,y}).
$$
Define
$$
\Phi_\varepsilon:M\to\mathcal{N}_\varepsilon,\qquad
\Phi_\varepsilon(y):=t_{\varepsilon,y}\varphi_{\varepsilon,y}.
$$

By the previous analysis on $c_\varepsilon\to c_0$ (see Lemma~\ref{lem:31} and Proposition~\ref{prop:33}), together with assumption $(W_2)$, there exist $\bar\varepsilon>0$ and a function
$$
\nu:(0,\bar\varepsilon)\to(0,\infty)
$$
such that
$$
\nu(\varepsilon)>c_\varepsilon-c_0,\qquad
\nu(\varepsilon)\to0\quad\text{as }\varepsilon\to0^+,
$$
and, for all $0<\varepsilon<\bar\varepsilon$ and $y\in M$,
$$
J_\varepsilon\bigl(\Phi_\varepsilon(y)\bigr)
=J_\varepsilon\bigl(t_{\varepsilon,y}\varphi_{\varepsilon,y}\bigr)
< c_0+\nu(\varepsilon)<c_\infty.
$$
Hence
$$
\Phi_\varepsilon(M)\subset\Sigma_\varepsilon,\qquad
0<\varepsilon<\bar\varepsilon.
$$

Recall the barycenter map $\beta_\varepsilon$ introduced above and observe that it is invariant under positive scalar multiplication, i.e.,
$$
\beta_\varepsilon(tv)=\beta_\varepsilon(v),\qquad
\forall\,t>0,\ v\in L^{p'}(\mathbb{R}^N)\setminus\{0\}.
$$
On the one hand, by \eqref{4-1},
$$
\beta_\varepsilon(\Phi_\varepsilon(y))
=\beta_\varepsilon\bigl(t_{\varepsilon,y}\varphi_{\varepsilon,y}\bigr)
=\beta_\varepsilon(\varphi_{\varepsilon,y})
\to y,
\qquad\text{as }\varepsilon\to0^+\text{ uniformly in }y\in M.
$$

On the other hand, by \eqref{4-2} there exists $\varepsilon_1\in(0,\bar\varepsilon]$ such that for all $0<\varepsilon<\varepsilon_1$,
$$
\beta_\varepsilon(\Sigma_\varepsilon)\subset M_\delta.
$$

Therefore, for sufficiently small $\varepsilon>0$ the composition
$$
\beta_\varepsilon\circ\Phi_\varepsilon:M\to M_\delta
$$
is well-defined and continuous, and converges uniformly to the identity map on $M$. More precisely,
$$
\sup_{y\in M}
\bigl|\beta_\varepsilon(\Phi_\varepsilon(y))-y\bigr|\to0
\quad\text{as }\varepsilon\to0^+.
$$
Thus there exists $\varepsilon_2\in(0,\varepsilon_1)$ such that, for all $0<\varepsilon<\varepsilon_2$,
$$
\sup_{y\in M}
\bigl|\beta_\varepsilon(\Phi_\varepsilon(y))-y\bigr|<\delta.
$$
Fix such an $\varepsilon$ and define the homotopy
$$
H:[0,1]\times M\to M_\delta,\qquad
H(t,y):=(1-t)\,\beta_\varepsilon(\Phi_\varepsilon(y))+t\,y.
$$
For any $t\in[0,1]$,
$$
|H(t,y)-y|
\le\bigl|\beta_\varepsilon(\Phi_\varepsilon(y))-y\bigr|
<\delta,
$$
and hence $H(t,y)\in B_\delta(y)\subset M_\delta$ for all $(t,y)$. Thus $H$ is a continuous homotopy with values in $M_\delta$ and satisfies
$$
H(0,y)=\beta_\varepsilon(\Phi_\varepsilon(y)),
\qquad
H(1,y)=y=i(y),
$$
where $i:M\hookrightarrow M_\delta$ is the natural inclusion. Consequently,
$$
\beta_\varepsilon\circ\Phi_\varepsilon\simeq i
\quad\text{in }M_\delta.
$$

By the category lemma of Cingolani and Lazzo (see  \cite[Lemma~2.2]{2000-Cingolani&Lazzo-JDE}), and using the above homotopy relation, we obtain
$$
\operatorname{cat}_{\mathcal{N}_\varepsilon}(\Sigma_\varepsilon)
\ge \operatorname{cat}_{M_\delta}(M)
=:m,
\qquad 0<\varepsilon<\varepsilon_2.
$$

Moreover, by $(W_1)$ and Lemma~\ref{lem:PS-below-cinf} we know that
$$
J_\varepsilon\big|_{\mathcal{N}_\varepsilon}
\ \text{satisfies the }(PS)\text{ condition on the energy interval }(-\infty,c_\infty).
$$
Since
$$
\sup_{\Sigma_\varepsilon}J_\varepsilon
\le c_0+\nu(\varepsilon)
<c_\infty,
$$
it follows that
$$
J_\varepsilon\big|_{\mathcal{N}_\varepsilon}
\ \text{also satisfies the }(PS)\text{ condition on }\Sigma_\varepsilon.
$$

Therefore, we can apply Lusternik--Schnirelmann theory to the functional
$$
J_\varepsilon\big|_{\mathcal{N}_\varepsilon}:\mathcal{N}_\varepsilon\to\mathbb{R}
$$
defined on the $C^1$-manifold $\mathcal{N}_\varepsilon$; for an Lusternik--Schnirelmann theory on $C^1$-manifolds we refer, for instance, to  \cite{1995-Ribarska&Tsachev&Krastanov-SMJ,1993-Corvellec&Degiovanni&Marzocchi-TMNA,1988-Szulkin-Ann}). From
$$
\operatorname{cat}_{\mathcal{N}_\varepsilon}(\Sigma_\varepsilon)\ge m
$$
we deduce that, for each $0<\varepsilon<\varepsilon_2$, there exist at least
$$
m:=\operatorname{cat}_{M_\delta}(M)
$$
pairwise distinct critical points
$$
v_\varepsilon^{(1)},\dots,v_\varepsilon^{(m)}\in\Sigma_\varepsilon\subset\mathcal{N}_\varepsilon
$$
such that
$$
J_\varepsilon'(v_\varepsilon^{(j)})=0,\qquad j=1,\dots,m.
$$

By the reconstruction formula in the dual framework (see Section~2.2, in particular \eqref{eq:dual-u} and the subsequent discussion), for each critical point
$$
v_\varepsilon\in\mathcal{N}_\varepsilon
$$
we define
$$
u_k(x)
:=k^{\frac{4}{p-2}}\,
\mathbf{R}\bigl(W_\varepsilon^{1/p}v_\varepsilon\bigr)(k x),
\qquad k=\varepsilon^{-1},\quad W_\varepsilon(x):=W(\varepsilon x),
$$
where $\mathbf{R}:L^{p'}(\mathbb{R}^N)\to L^p(\mathbb{R}^N)$ is the Helmholtz resolvent operator. Then $u_k$ is a nontrivial solution of the original problem
$$
\Delta^2 u - \beta k^2\Delta u + \alpha k^4 u
= W(x)\,|u|^{p-2}u
\quad\text{in }\mathbb{R}^N,
$$
and there is a one-to-one correspondence between nontrivial critical points of $J_\varepsilon$ and nontrivial solutions of the above equation; see \eqref{eq:dual-u}  and the subsequent discussion. In particular, distinct $v_\varepsilon^{(j)}$ give rise to distinct $u_k^{(j)}$.

Applying this reconstruction to each $v_\varepsilon^{(j)}$, for every $0<\varepsilon<\varepsilon_2$ we obtain
$$
u_k^{(j)}(x)
:=k^{\frac{4}{p-2}}\,
\mathbf{R}\bigl(W_\varepsilon^{1/p}v_\varepsilon^{(j)}\bigr)(k x),
\qquad j=1,\dots,m,
$$
where $k=\varepsilon^{-1}$. If $k>k(\delta):=\varepsilon_2^{-1}$, then problem \eqref{1.5} admits at least
$$
m=\operatorname{cat}_{M_\delta}(M)
$$
pairwise distinct nontrivial solutions $u_k^{(1)},\dots,u_k^{(m)}$. This proves the multiplicity part of Theorem~1.2.

\medskip
\emph{Concentration behavior of solutions.}
We now fix some $j\in\{1,\dots,m\}$ and study the concentration behaviour of the corresponding family $u_k^{(j)}$.

Since
$$
v_\varepsilon^{(j)}\in\Sigma_\varepsilon
=\{v\in\mathcal{N}_\varepsilon:\ J_\varepsilon(v)\le c_0+\nu(\varepsilon)\},
$$
and
$$
c_\varepsilon\le J_\varepsilon\bigl(v_\varepsilon^{(j)}\bigr)\le c_0+\nu(\varepsilon),
$$
from $c_\varepsilon\to c_0$ and $\nu(\varepsilon)\to0$ we infer that, for each fixed $j$,
$$
J_\varepsilon\bigl(v_\varepsilon^{(j)}\bigr)\to c_0,
\qquad\varepsilon\to0^+.
$$

Let $(k_n)_n$ be any sequence with $k_n>k(\delta)$ and $k_n\to\infty$, and set
$$
\varepsilon_n:=k_n^{-1}\to0^+,\qquad
v_n:=v_{\varepsilon_n}^{(j)}\in\mathcal{N}_{\varepsilon_n}.
$$
Then
$$
J_{\varepsilon_n}(v_n)\to c_0.
$$

By Proposition~\ref{prop:34}, there exist a point $x_0^{(j)}\in M$, a least energy critical point $w_0^{(j)}$ of the limit functional $J_0$, and a sequence of translation vectors $(y_n^{(j)})_n\subset\mathbb{R}^N$ such that, up to a subsequence,
$$
x_n^{(j)}:=\varepsilon_n y_n^{(j)}\to x_0^{(j)},
\qquad
v_n(\cdot+y_n^{(j)})\to w_0^{(j)}
\quad\text{in }L^{p'}(\mathbb{R}^N).
$$
Define
$$
u_n(x):=u_{k_n}^{(j)}(x)
=k_n^{\frac{4}{p-2}}\,
\mathbf{R}\bigl(W_{\varepsilon_n}^{1/p}v_n\bigr)(k_n x),
\qquad W_{\varepsilon_n}(x):=W(\varepsilon_n x).
$$
We claim that
$$
k_n^{-\frac{4}{p-2}}\,
u_n\Bigl(\tfrac{x}{k_n}+x_n^{(j)}\Bigr)
\to u_0^{(j)}
\quad\text{in }L^p(\mathbb{R}^N),
$$
for some limit function $u_0^{(j)}$. A direct computation shows that
$$
\begin{aligned}
	k_n^{-\frac{4}{p-2}}\,
	u_n\Bigl(\tfrac{x}{k_n}+x_n^{(j)}\Bigr)
	&=\mathbf{R}\bigl(W_{\varepsilon_n}^{1/p}v_n\bigr)
	\Bigl(x+k_n x_n^{(j)}\Bigr)\\
	&=\mathbf{R}\bigl(W_{\varepsilon_n}^{1/p}v_n\bigr)
	\bigl(x+y_n^{(j)}\bigr)\\
	&=\mathbf{R}\Bigl(
	W_{\varepsilon_n}^{1/p}(\cdot+y_n^{(j)})\,v_n(\cdot+y_n^{(j)})
	\Bigr)(x).
\end{aligned}
$$
Let
$$
g_n(\xi):=
W_{\varepsilon_n}^{1/p}(\xi+y_n^{(j)})\,v_n(\xi+y_n^{(j)}),
\qquad \xi\in\mathbb{R}^N,
$$
so that
$$
k_n^{-\frac{4}{p-2}}\,
u_n\Bigl(\tfrac{\cdot}{k_n}+x_n^{(j)}\Bigr)
=\mathbf{R}(g_n).
$$

We now prove that $g_n\to W_0^{1/p}w_0^{(j)}$ in $L^{p'}(\mathbb{R}^N)$, where $W_0:=W(x_0^{(j)})$. Notice that
$$
W_{\varepsilon_n}(\xi+y_n^{(j)})
=W\bigl(\varepsilon_n(\xi+y_n^{(j)})\bigr)
=W\bigl(x_n^{(j)}+\varepsilon_n\xi\bigr),
$$
and, for each fixed $\xi$, $x_n^{(j)}+\varepsilon_n\xi\to x_0^{(j)}$ as $n\to\infty$. Hence
$$
W_{\varepsilon_n}^{1/p}(\xi+y_n^{(j)})
\to W_0^{1/p}
\quad\text{pointwise in }\mathbb{R}^N.
$$                                                                                 

On the other hand, $v_n(\cdot+y_n^{(j)})\to w_0^{(j)}$ in $L^{p'}(\mathbb{R}^N)$, so the sequence $(v_n(\cdot+y_n^{(j)}))_n$ is bounded in $L^{p'}(\mathbb{R}^N)$. For any $R>0$, $W_{\varepsilon_n}^{1/p}(\cdot+y_n^{(j)})\to W_0^{1/p}$ converges uniformly in $L^\infty(B_R)$, and hence
$$
\bigl\|(W_{\varepsilon_n}^{1/p}(\cdot+y_n^{(j)})-W_0^{1/p})\,w_0^{(j)}\bigr\|_{L^{p'}(B_R)}
\to0.
$$
Combining this with $v_n(\cdot+y_n^{(j)})\to w_0^{(j)}$ in $L^{p'}(\mathbb{R}^N)$, and using the triangle inequality and Hölder's inequality, we obtain
$$
\bigl\|g_n-W_0^{1/p}w_0^{(j)}\bigr\|_{L^{p'}(B_R)}\to0.
$$

On $B_R^c:=\mathbb{R}^N\setminus B_R(0)$, we use the facts that $w_0^{(j)}\in L^{p'}(\mathbb{R}^N)$ and $v_n(\cdot+y_n^{(j)})\to w_0^{(j)}$ strongly in $L^{p'}(\mathbb{R}^N)$. Given $\eta>0$, choose $R>0$ large enough so that
$$
\|w_0^{(j)}\|_{L^{p'}(B_R^c)}<\eta.
$$
Then, by the $L^{p'}$-convergence, there exists $n_0\in\mathbb{N}$ such that for all $n\ge n_0$,
$$
\|v_n(\cdot+y_n^{(j)})-w_0^{(j)}\|_{L^{p'}(\mathbb{R}^N)}<\eta.
$$
In particular,
$$
\|v_n(\cdot+y_n^{(j)})\|_{L^{p'}(B_R^c)}
\le \|w_0^{(j)}\|_{L^{p'}(B_R^c)}
+\|v_n(\cdot+y_n^{(j)})-w_0^{(j)}\|_{L^{p'}(B_R^c)}
<2\eta,
$$
for all $n\ge n_0$.                                                                 
Since $\|W_{\varepsilon_n}^{1/p}(\cdot+y_n^{(j)})\|_{L^\infty(\mathbb{R}^N)}\le C$, we obtain
$$
\|g_n\|_{L^{p'}(B_R^c)}
\le C\|v_n(\cdot+y_n^{(j)})\|_{L^{p'}(B_R^c)}
\le 2C\eta,\qquad n\ge n_0.
$$
Moreover,
$$
\|W_0^{1/p}w_0^{(j)}\|_{L^{p'}(B_R^c)}
=W_0^{1/p}\,\|w_0^{(j)}\|_{L^{p'}(B_R^c)}
\le C_0\eta
$$
for some constant $C_0>0$ independent of $n$. Thus, for all $n\ge n_0$,
$$
\bigl\|g_n-W_0^{1/p}w_0^{(j)}\bigr\|_{L^{p'}(B_R^c)}
\le \|g_n\|_{L^{p'}(B_R^c)}+\|W_0^{1/p}w_0^{(j)}\|_{L^{p'}(B_R^c)}
\le (2C+C_0)\eta.
$$
Since $\eta>0$ is arbitrary, for $R$ sufficiently large we have
\[
\bigl\| g_n - W_0^{1/p} w_0^{(j)} \bigr\|_{L^{p'}(B_R^c)}
\;\longrightarrow\; 0
\quad \text{as } n \to \infty .
\]
Together with the convergence on $B_R$, we conclude that
$$
g_n
=W_{\varepsilon_n}^{1/p}(\cdot+y_n^{(j)})\,v_n(\cdot+y_n^{(j)})
\to W_0^{1/p}w_0^{(j)}
\quad\text{in }L^{p'}(\mathbb{R}^N).
$$

Since $\mathbf{R}:L^{p'}(\mathbb{R}^N)\to L^p(\mathbb{R}^N)$ is continuous, we obtain
$$
k_n^{-\frac{4}{p-2}}\,
u_n\Bigl(\tfrac{\cdot}{k_n}+x_n^{(j)}\Bigr)
=\mathbf{R}(g_n)
\to
\mathbf{R}\bigl(W_0^{1/p}w_0^{(j)}\bigr)
=:u_0^{(j)}
\quad\text{in }L^p(\mathbb{R}^N).
$$

By the dual variational theory developed in Section~2, $u_0^{(j)}$ is a ground state solution of the limit problem
$$
\Delta^2 u - \beta \Delta u + \alpha u
=W_0\,|u|^{p-2}u
\quad\text{in }\mathbb{R}^N.
$$
Since $x_n^{(j)}\to x_0^{(j)}\in M$, we obtain precisely the concentration property stated in Theorem~1.2, namely,
$$
x_n^{(j)}\to x_0^{(j)},\qquad
k_n^{-\frac{4}{p-2}}\,
u_{k_n}^{(j)}\Bigl(\tfrac{\cdot}{k_n}+x_n^{(j)}\Bigr)
\to u_0^{(j)}
\quad\text{in }L^p(\mathbb{R}^N).
$$

This completes the proof of Theorem~1.2.
\qed

\begin{Rem}
	According to the regularity results in \cite[Propositions~5.1 and 5.2]{2019-Denis-Jean-Rainer-JLMS}, one has
	$$
	u_0^{(j)}\in W^{4,q}(\mathbb{R}^N)\quad\text{for all }1<q<\infty.
	$$
	Using elliptic regularity for the limit equation and the uniform boundedness of $u_{k_n}^{(j)}$ in $W^{4,q}$, one can further show that there exists a subsequence (still denoted by $k_n$) such that, for every $1<q<\infty$,
	$$
	k_n^{-\frac{4}{p-2}}\,
	u_{k_n}^{(j)}\!\Bigl(\tfrac{\cdot}{k_n}+x_n^{(j)}\Bigr)
	\ \to\ u_0^{(j)}
	\quad \text{in }W^{4,q}(\mathbb{R}^N).
	$$
	Taking $q>\dfrac{N}{4}$ and using the Sobolev embedding $W^{4,q}(\mathbb{R}^N)\hookrightarrow L^\infty(\mathbb{R}^N)$, we obtain
	$$
	k_n^{-\frac{4}{p-2}}\,
	\|u_{k_n}^{(j)}\|_{L^\infty(\mathbb{R}^N)}
	\;\to\;
	\|u_0^{(j)}\|_{L^\infty(\mathbb{R}^N)}\;>\;0.
	$$
	Moreover, from the decay of $u_0^{(j)}$ at infinity (namely, for every $\delta>0$ there exists $R_\delta>0$ such that
	$|u_0^{(j)}(x)|<\delta$ whenever $|x|\ge R_\delta$), together with the above convergence in $W^{4,q}$ and $L^\infty$, it follows that for all sufficiently large $n$,
	$$
	k_n^{-\frac{4}{p-2}}\,
	\bigl|u_{k_n}^{(j)}(x)\bigr|
	<\delta
	\quad\text{whenever }|x-x_n^{(j)}|
	\ge\frac{R_\delta}{k_n}.
	$$
	Hence the modulus of $u_{k_n}^{(j)}$ attains its global maximum only in an $O(k_n^{-1})$-neighbourhood of $x_n^{(j)}$. If $\tilde{x}_n^{(j)}$ denotes any global maximum point of $\bigl|u_{k_n}^{(j)}\bigr|$, then necessarily
	$$
	\bigl|\tilde{x}_n^{(j)}-x_n^{(j)}\bigr|
	\le\frac{R_\delta}{k_n}.
	$$
	Letting $n\to\infty$ (and recalling that $x_n^{(j)}\to x_0^{(j)}$), we conclude that
	$$
	\tilde{x}_n^{(j)}\to x_0^{(j)}\in M,
	$$
	that is, the global maxima of these multiple solutions concentrate around $M$.
\end{Rem}

\vspace{1em}

\par
{\bf Conflict Of Interest Statement.} The authors declare that there are no conflict of interests, we do not have any possible conflicts of interest.
\par
{\bf Data Availability Statement.} Our manuscript has non associated data.
\bibliographystyle{plain}
\bibliography{references}
\end{document}